\documentclass{article}
\usepackage{amssymb}
\usepackage{amsmath}
\usepackage{amsthm} 
\usepackage{multirow}
\usepackage[english]{babel}

\usepackage{graphicx,wrapfig}
\usepackage{array}

\usepackage{float}
\usepackage{xcolor}
\usepackage{subcaption}
\usepackage{hyperref}

\usepackage[margin=0.25in,headheight = 0.0in, footskip=0.2in]{geometry}  

\newtheorem{theorem}{Theorem}
\newtheorem{conclusion}[theorem]{Conclusion}
\newtheorem{definition}{Definition}
\newtheorem{lemma}[theorem]{Lemma}
\newtheorem{rem}[theorem]{Remark}

\usepackage{pst-plot}
\usepackage{tikz}
\usetikzlibrary{arrows}
\tikzstyle{arrow} = [thin,->,>=stealth]
\usetikzlibrary{er,positioning}
\usetikzlibrary{patterns}

\colorlet{lightRed}{red!15}
\colorlet{lightGrey}{black!8}
\colorlet{veryLightGrey}{black!3}

\usepackage[shortlabels]{enumitem}

\newcommand{\rank }{\,{\rm rank}}

\def\reals{\mathbb{R}}
\def\eps{\varepsilon}
\def\bigo{{\mathcal O}}
\def\vres{V_{\rm res}}
\def\ncl{n_{\rm cl}}
\def\RH{r_{\rm H}}
\def\bbar{\overline}
\def\C{\mathcal P}
\def\qnm{Q_N^m}
\def\pkm{P_k^\mu}
\def\OX{the $x$ axis}
\def\OY{the $y$ axis}

\def\RS{${\cal RS}$ }

\newcommand{\cc}[1]{\mbox{CC${}_{#1}$}}
\newcommand{\re}[1]{{\rm RE}(#1)}
\newcommand{\ccrbp}[1]{\mbox{\rm rbp${}{(#1, a, b)}$}}

\title{Central configurations on the plane with  $N$ heavy and $k$ light bodies}

\author{Ma{\l}gorzata Moczurad and Piotr Zgliczy\'nski\footnote{Partially supported by the NCN grant  2019/35/B/ST1/00655}\ \footnote{Corresponding author}\\
   \{malgorzata.moczurad, piotr.zgliczynski\}@uj.edu.pl \\
Faculty of Mathematics and Computer Science, Jagiellonian University,\\
ul. prof. Stanis\l awa \L ojasiewicza 6,
30-348 Krak\'ow, Poland
}

\begin{document}
\maketitle

\begin{abstract}We study the problem of planar central configurations with $N$ heavy bodies and $k$ bodies with arbitrary small masses.
We derive the equation which describe the limit of light masses going to zero, which can be seen as the equation for central configurations in the anisotropic plane.  Using computer rigorous computations we compute
all central configurations for $N=2$ and $k=3,4$ and for the derived limit problem. We show that the results are consistent.
\end{abstract}

\tableofcontents


\section{Introduction}
The main problem of the classical celestial mechanics is the $n$-body problem, i.e. the description of the motion of $n$ point masses under their mutual Newtonian gravitational forces. This problem is entirely solved for $n = 2$, while for $n\ge 2$ only partial results exist.

A {\em central configuration}\index{central configuration}, denoted here as CC, is an initial configuration, such that if the particles were all released with zero velocity, they would collapse toward the center of mass at the same time.
In the planar case, CCs are initial positions for periodic solutions of the $n$-body problem where bodies move on a Keplerian elliptical orbits preserving the shape of the configuration.  The circular orbits give rise to solutions for which the configuration rotates at constant angular speed. 
This is called a relative equilibrium solution (see Figure~\ref{fig:homograph}). The other orbits 
give rise to solutions where radius and the angular speed are not constant.

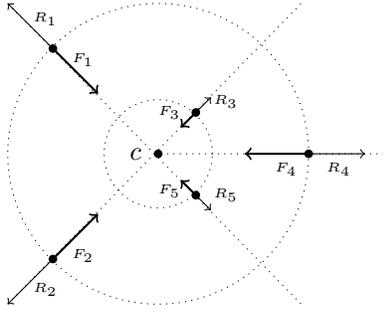
\begin{figure}[H]
  \centering
  \begin{tikzpicture}

\draw[dotted, thin] (1.5, 4.0) -- (5.4, 0.0);
\draw[dotted, thin] (1.5, 0.0) -- (5.4, 4.0);

\draw[dotted, thin] (6.5, 2.0) -- (3.5, 2.0);

\draw[fill] (3.5, 2.0) circle (0.05cm);
\node[] at (3.2, 2.0) {$c$};
\draw[dotted] (3.5, 2.0) circle (2.0cm);
\draw[dotted] (3.5, 2.0) circle (0.72cm);

\node[] at (2.0, 3.8) {\tiny $R_1$};
\node[] at (2.5, 3.25) {\tiny $F_1$};
\draw[->, thick] (2.1, 3.4) -- (2.7, 2.8);
\draw[->, thin] (2.1, 3.4) -- (1.5, 4.0);
\draw[fill] (2.1, 3.4) circle (0.05cm);

\node[] at (2.0, 0.2) {\tiny $R_2$};
\node[] at (2.5, 0.65) {\tiny $F_2$};
\draw[->,  thick] (2.1, 0.6) -- (2.7, 1.2);
\draw[->,  thin] (2.1, 0.6) -- (1.5, 0.0);
\draw[fill] (2.1, 0.6) circle (0.05cm);

\node[] at (4.4, 2.7) {\tiny $R_3$};
\node[] at (3.65, 2.55) {\tiny $F_3$};
\draw[->, thick] (4.0, 2.55) -- (3.8, 2.35);
\draw[->, thin] (4.0, 2.55) -- (4.2, 2.75);
\draw[fill] (4.0, 2.55) circle (0.05cm);

\node[] at (5.9, 1.8) {\tiny $R_4$};
\node[] at (5.2, 1.8) {\tiny $F_4$};
\draw[->, thick] (5.5, 2.0) -- (4.65, 2.0);
\draw[->, thin] (5.5, 2.0) -- (6.25, 2.0);
\draw[fill] (5.5, 2.0) circle (0.05cm);

\node[] at (3.65, 1.5) {\tiny $F_5$};
\node[] at (4.4, 1.45) {\tiny $R_5$};
\draw[->,  thick] (4.0, 1.45) -- (3.8, 1.65);
\draw[->,  thin] (4.0, 1.45) -- (4.2, 1.25);
\draw[fill] (4.0, 1.45) circle (0.05cm);

\end{tikzpicture}
\caption{If we give the particles an appropriate angular velocity such that gravitational forces $F_i$ balance (equilibrate) centrifugal forces $R_i$, they will move along circular orbits.}\label{fig:homograph}
\end{figure}

\subsection{The state of the art}
It is conjectured  in \cite{Wintner}  that there are finitely many  isometry classes of CCs. The conjecture appears as the sixth
problem of Smale's eighteen problems for the 21st century \cite{SmNext}. We refer to this problem in short as {\em CC finiteness}. From among abundant literature on the subject, let us mention two important works~\cite{HM} and \cite{AK} which aim at studying the problem in the most general setting and provide results for $n=4$ and $n=5$ with (almost) arbitrary masses.  
In \cite{HM} the authors give a computer assisted proof of CC finiteness for $n=4$ for any system of positive masses; in \cite{AK} the problem for $n=4$ is proven without computer assistance and the finiteness for $n=5$ is proven for arbitrary positive masses with the possible exception of those CCs where quintuples of positive masses belongs to a
given codimension 2 subvariety of the mass space. 
It is interesting to note that 
when masses are equal, there are finitely many isometry classes of CCs and this case belongs to the aforementioned subvariety (see~\cite{MZ}). 
A common feature of these works
is that they give a quite poor estimate for the maximum number of CCs. In this context it is worth to mention the work of Simo,  based on extensive numerical studies. In~\cite{Si78} he gives  the number of CCs for all possible masses for $n=4$.

The papers \cite{AK,HM,M01,HJ} which are concerned with CC finiteness for all masses study the polynomial equations derived from the equations for CC  using the algebraic geometry tools. 
The approach we pursue is different as we study the equations for CCs using rigorous numerics. 
While our results published so far (see~\cite{MZ,MZ20}) refer to the equal masses case, our program allows to process boxes in the mass space, hence in principle we can cover the whole mass space and find the number of CCs with all masses.  However, there are several obstacles in the case of different masses: some of the masses approaching zero and possible bifurcations. In the current paper, we try to overcome the former problem.

\subsection{Our contribution}
We look for CCs in the situation when the massess of certain bodies approach zero as the positions of these bodies move closer together (approaching the same point).  We derive equations for this limit problem, which can be seen as the equations for {\em central configurations in the anisotropic plane}. Our results can be seen as a generalization of the results by Xia \cite{X}, where the case of two light bodies was considered.   We perform computer assisted proofs of full listings of CCs in the anisotropic plane
for $k=3,4$. The anisotropy parameters correspond to CCs for two heavy bodies with equal masses and $k$ light bodies with very small equal masses
in the neighborhood of Lagrange point $L_4$ of the restricted two body problem. We also perform rigorous  computations of all CCs
for $(2+k)$-body problem in the corresponding region of configuration space. The results of both computations are in very good agreement.

The content of the paper is as follows. In Section~\ref{sec:cc-eq} we recall equations for central configurations. In Section~\ref{sec:red-sys-equiv} we derive the reduced set of equations for CCs where we remove the rotational symmetry and
the system of equations becomes amenable to a direct application of the implicit function theorem.  In Section~\ref{sec:cc-N+k} the limit of vanishing masses is studied, the limit problem is derived and basic theorem about the continuation of solutions of the limit problem is proved.
In Section~\ref{sec:restr-cc-light-bodies} we study general properties of the limit problem and we establish some a priori bounds. 
In Section~\ref{sec:analytic-sol} we find analytically several solutions of the limit problem. In Section~\ref{sec:cap}  we present computer assisted
proofs for the limit problem with $k=3,4$ and corresponding CCs for $(2+k)$-bodies with $k=3,4$.

Proofs which are either very simple, or analogous to those in~\cite{MZ} are contained in Appendix~\ref{proofs}.

\subsection{Notation}
Throughout the paper we use the following notation:
\begin{itemize}
\item body --- point characterized by its position $q\in\reals^2$ and the mass $m$;
\item CC -- central configuration
\item $\cc{n}$ -- central configuration of $n$ point masses (bodies); $q\in\cc{n}$ --- $q = (q_1, \ldots, q_n)$ is a central configuration of $n$ bodies;
\item $r_{ij} = |q_i  - q_j|$, where $q_i, q_j$ are positions of bodies.
\end{itemize}
Usually we use letters $p$ and $q$ for light and heavy bodies, respectively.  Uppercase $P$ and $Q$ denote their respective configurations, with subscripts and superscripts indicating the number of bodies and vector of masses. Both these parameters are consider constant when solving the equations for central configurations (see equation~(\ref{eq:cc-with-lambda})).
\begin{itemize}
\item  $\qnm$ and $P_k^{\mu}$, where
\begin{itemize}
\item $(q_1, \ldots, q_N)$ --- configuration of heavy bodies with masses $m_1, \ldots, m_N$, respectively; we use $m = (m_i)_{i = 1, \ldots, N}$
\item $(p_1, \ldots, p_k)$ --- configuration of light bodies with masses $\mu_1, \ldots, \mu_k$, respectively; use $\mu = (\mu_i)_{i = 1, \ldots, k}$
\item $x\in \re{Q_N}$ --- $x$ is a {\em relative normalized CC in the restricted $(N+1)$-body problem} for the configuration $Q_N$ (see Definition~\ref{def:rel-equi-restricted});
\item $\cc{}(W, \qnm, x, \mu)$ --- central configuration of potential  $W$  induced by $(Q_N, x)$ for $\mu$ (see Definition~\ref{def:cc-induced}).
\end{itemize}
If we indicate masses explicitly we sometimes omit superscripts and just we write $Q_N$ or $P_k$.

\end{itemize}

When we focus only on light bodies we use:
\begin{itemize}
\item \ccrbp{k}  -- CC of restricted $k$ light body problem (see Definition\ref{def:ccrbp}), solutions of the anisotropic central configuration
  problem (\ref{eqn:a-b}).
\end{itemize}


\section{Central configurations}
\label{sec:cc-eq}

Assume there are a group $\qnm \in \left(\mathbb{R}^2\right)^N$ of $N$ celestial bodies  interacting with each other gravitationally (i.e.\ due to inverse square Law of Gravitation; gravitational constant is normalized $G = 1$).
The {\em central configuration problem}   is to find positions of bodies  satisfying the following system of equations
\begin{equation}
  \lambda m_i(q_i-c) = \sum_{\substack{j=1\\
j\neq i}}^N \frac{m_im_j}{r_{ij}^3}(q_i - q_j)=: f_i(q_1,\dots,q_n), \quad i=1,\dots,n, \label{eq:cc-with-lambda}
\end{equation}
where $c$ is a center of mass of the configuration and $\lambda \in \mathbb{R}_+$ is a scaling factor.
It is well-known (see for example \cite{Mlect2014,MZ,AK}) that in equations~(\ref{eq:cc-with-lambda}) we can set $c=0$ obtaining
\begin{equation}
  -\lambda m_i q_i = \sum_{j \neq i} \frac{m_i m_j}{|q_i - q_j|^3} (q_j-q_i)=\frac{\partial U_N}{\partial q_i},\quad i=1,2,\dots, N \label{eq:cc}
\end{equation}
where
\begin{equation*}
  U_N=\sum_{1\leq i<j\leq N} \frac{m_i m_j}{|q_i - q_j|}.
\end{equation*}

\begin{definition}\cite{AK,MZ}
A \emph{normalized central configuration} $Q_N \in \left(\mathbb{R}^2\right)^N$ is a central configuration with $\lambda=1$ and $c=0$.
\end{definition}
From now on we focus on normalized central configurations.

We can reformulate equations (\ref{eq:cc}) with $\lambda=1$ as the problem of finding critical points for certain potential function.
Let us define potential
\begin{equation}
  V(q_1,\dots,q_N)= \sum_{i = 1}^N \frac{m_i q_i^2}{2} + U_N(q_1,\dots,q_N) = \sum_{i = 1}^N \frac{m_i q_i^2}{2} + \sum_{1\leq i<j\leq N} \frac{m_i m_j}{|q_i - q_j|}.   \label{eq:cc-pot}
\end{equation}
Note that for $\lambda=1$ the equations~(\ref{eq:cc}) become
 \begin{eqnarray}
 \frac{\partial V}{\partial q_i}(q_1,\dots,q_N) = 0, \quad i=1,\dots,N.   \label{eq:cc-DV}
 \end{eqnarray}

\section{The reduced system of equations for CC }
\label{sec:red-sys-equiv}

 The goal of this section is to derive a set of equations ({\em the reduced system of equations}), which gives all (equivalence classes of) CCs and which no longer has $\mathcal{SO}(d)$-symmetry, where $d$ is the dimension of the space. In our case $d = 2$ that means that we eliminate rotations around the origin. This section is an extension of the results of Section 5 in \cite{MZ} and Section 3 in \cite{MZ20}.

For the future use we introduce the function $F:\Pi_{i=1}^n\mathbb{R}^{d } \to \Pi_{i=1}^n\mathbb{R}^{d }$ given by
\begin{equation}\label{eq:vector-field}
  F_i(q_1,\dots,q_n) =  q_i - \sum_{j,j\neq i} \frac{m_j}{r_{ij}^3}(q_i - q_j), \quad i=1,\dots,n.
  \index{$F_i$}
\end{equation}
Then the system for normalized central configurations
\begin{equation}
  q_i= \sum_{j,j\neq i} \frac{m_j}{r_{ij}^3}(q_i - q_j)=:\frac{1}{m_i}f_i(q_1,\dots,q_n), \quad i=1,\dots,n.
  \index{$q_i$} \label{eq:cc-kart}
\end{equation}
becomes
\begin{equation}\label{eq:cc-abstract}
  F(q_1,\dots,q_n)=0.
  \index{$F$}
\end{equation}
It is well known (see \cite{MZ} and the literature given there) that for any $(q_1,\ldots,q_n)\in (\reals^d)^n$ holds
\begin{eqnarray}
\sum_{i=1}^n f_i&=&0, \label{eq:n-mom-con} \\
\sum_{i=1}^n f_i \wedge q_i & = & 0, \label{eq:n-angular-mom-con}
\end{eqnarray}
where $v \wedge w$\index{$v \wedge w$} is the exterior product of vectors, the result being an element of exterior algebra. If $d=2$ or 3 it can be interpreted as the vector product of $v$ and $w$ in dimension $3$.  The identities (\ref{eq:n-mom-con}) and (\ref{eq:n-angular-mom-con})
are easy consequences of the third Newton's law (the action equals reaction) and the requirement that the mutual forces between bodies are in direction of the other body.

Consider system (\ref{eq:cc-kart}). After multiplication of  $i$-th equation by $m_i$ and addition of all equations using (\ref{eq:n-mom-con}) we obtain (or rather recover)
the center of mass equation
\begin{eqnarray}
 \left(\sum_{i=1}^n m_i\right) c=\sum_i m_i q_i = 0. \label{eq:cc-cofmass-l}
\end{eqnarray}
We can take the equations for $n$-th body and replace it with (\ref{eq:cc-cofmass-l}) to obtain an equivalent system
\begin{subequations}
\begin{align}
  q_i&= \sum_{j,j\neq i} \frac{m_j}{r_{ij}^3}(q_i - q_j), \quad i=1,\dots,n-1, \label{eq:cc-kart-1n-1} \\
   q_n&=-\frac{1}{m_n}\sum_{i=1}^{n-1} m_i q_i. \label{eq:cc-kart-n-th}
   \index{$q_i$}
\end{align}\label{eqn:reduced}
\end{subequations}

\subsection{Non-degenerate solutions of full and reduced systems}

Following Moeckel \cite{Mlect2014} we state the following definition for any $d$, but we are interested in $d=2$, here.
\begin{definition}
  \label{def:non-deg-cc} We say that a normalized central configuration $q=(q_1,\dots,q_n)$ is \emph{non-degenerate}
  if the rank of $D\!F(q)$ is equal to $dn-\dim \mathcal{SO}(d)$. Otherwise the configuration is called \emph{degenerate}.
\end{definition}
The idea of the above notion of degeneracy is to allow only for the  degeneracy related to the rotational symmetry of the problem, because
by setting $\lambda=1$ in (\ref{eq:cc-with-lambda}) and keeping the masses fixed we removed the scaling symmetry.

\subsubsection{The center of mass reduction}
\label{subsubsec:com-red}

We  write the  system (\ref{eqn:reduced}) obtained from (\ref{eq:cc-abstract}) after removing the $n$-th body using the center of mass equation (condition (\ref{eq:cc-cofmass-l})) as
\begin{equation}\label{eq:cc-abstract-red}
  F_{\mathrm{red}}(q_1,\dots,q_{n-1})=0,
\end{equation}
where $F_{\mathrm{red}}: \Pi_{i=1}^{n-1}\mathbb{R}^{d} \to \Pi_{i=1}^{n-1}\mathbb{R}^{d} $.
To be precise we have
\begin{eqnarray*}
   F_{\mathrm{red},i}(q_1,\dots,q_{n-1})&=& F_i(q_1,\dots,q_{n-1},q_n(q_1,\dots,q_{n-1})),\quad i=1,\dots,n-1
\end{eqnarray*}
where
\begin{equation}
    q_n(q_1,\dots,q_{n-1})=-\frac{1}{m_n}\sum_{i=1}^{n-1} m_i q_i. \label{eq:cc-com}
\end{equation}

\begin{lemma}
\label{lem:rank-com-red}
If $q=(q_1,\dots,q_{n-1},q_n)$ is a normalized CC. Then
\begin{equation}
  \rank\left(D\! F_{\mathrm{red}}(q_1,\dots,q_{n-1})\right) = \rank\left(D\! F(q_1,\dots,q_{n-1},q_n)\right)-d.
  \end{equation}
\end{lemma}
An easy proof is left to reader.

For any configuration $q$ we set
\begin{equation}\label{eq:Ri}
  R_i(q_1,\dots,q_n)=m_iq_i - f_i(q_1,\dots,q_n), \quad i=1,\dots,n.
\end{equation}
With the above notation the system (\ref{eq:cc-abstract}) becomes
\begin{equation}
  R(q_1,\dots,q_n)=
  (R_1(q_1,\dots,q_n),\dots,R_n(q_1,\dots,q_n))=0. \label{eq:cc-R}
\end{equation}
For any $(q_1,\dots,q_{n-1}) \in (\mathbb{R}^2)^{n-1}$
we define
\begin{eqnarray}\label{eq:R-tylda}
  \tilde{R}_i(q_1,\dots,q_{n-1})=R_i(q_1,\dots,q_{n-1},q_n(q_1,\dots,q_{n-1})), \quad i=1,\dots,n.
\end{eqnarray}
Obviously $\rank(D\tilde{R})(q) = \rank(DF_{\mathrm{red}})(q) $.
Observe that for any $(q_1,\dots,q_{n-1}) \in (\mathbb{R}^d)^{n-1}$  holds
\begin{eqnarray}
  \tilde{R}_n(q_1,\dots,q_{n-1})=-\sum_{i=1}^{n-1}\tilde{R}_i(q_1,\dots,q_{n-1}). \label{eq:cc-com-red}
\end{eqnarray}
Indeed, by~(\ref{eq:cc-kart}) we have
\begin{eqnarray*}
\sum_{i = 1}^{n-1}f_i(q_1, \ldots, q_{n-1}, q_n(q_1,\dots,q_{n-1})) & = & \sum_{i = 1}^{n-1} m_iq_i
\end{eqnarray*}
and by~(\ref{eq:cc-com})
\begin{eqnarray*}
q_n(q_1, \ldots, q_{n-1}, q_n(q_1,\dots,q_{n-1})) & = & - \frac{1}{m_n} \sum_{i = 1}^{n-1} m_iq_i \\
  & = & -\frac{1}{m_n}\sum_{i = 1}^{n-1}f_i(q_1, \ldots, q_{n-1}, q_n(q_1,\dots,q_{n-1}))
\end{eqnarray*}
thus
\begin{eqnarray}
f_n(q_1, \ldots, q_{n-1}, q_n(q_1, \ldots, q_{n-1})) & = & m_nq_n(q_1, \ldots, q_{n-1}, q_n(q_1, \ldots, q_{n-1})) \nonumber\\
& = & - \sum_{i = 1}^{n-1}f_i(q_1, \ldots, q_{n-1}, q_n(q_1, \ldots, q_{n-1}))\label{eqn:fRed}
\end{eqnarray}
Now from~(\ref{eq:R-tylda}) and (\ref{eqn:fRed}) we obtain
\begin{eqnarray*}
 \tilde{R}_n(q_1,\dots,q_{n-1}) & = & R_n(q_1,\dots,q_{n-1},q_n(q_1,\dots,q_{n-1})) \\
 & = &  m_n q_n(q_1,\dots,q_{n-1})-f_n(q_1,\dots,q_{n-1},q_n(q_1,\dots,q_{n-1}))  \\
 & = & -\sum_{i=1}^{n-1} m_i q_i + \sum_{i=1}^{n-1}f_i(q_1,\dots,q_{n-1},q_n(q_1,\dots,q_{n-1})) \\
 & = & - \sum_{i=1}^{n-1}\tilde{R}_i(q_1,\dots,q_{n-1}).
\end{eqnarray*}
Observe that from (\ref{eq:n-angular-mom-con}) it follows that for any configuration $(q_1,\dots,q_n)$ holds
\begin{equation}
  \sum_{i=1}^n q_i \wedge R_i(q_1,\dots,q_n)=0.  \label{eq:qiwedgeRi}
\end{equation}
In particular for
 $q_n=q_n(q_1,\dots,q_{n-1})$ we obtain from (\ref{eq:qiwedgeRi}) and (\ref{eq:cc-com-red})
\begin{eqnarray}
  0 & = & \sum_{i=1}^{n-1}q_i \wedge \tilde{R}_i \left(q_1,\dots,q_{n-1}\right) + q_n(q_1,\dots,q_{n-1}) \wedge \tilde{R}_n \left(q_1,\dots,q_{n-1}\right) \nonumber\\
  & = &
  \sum_{i=1}^{n-1}(q_i-q_n) \wedge \tilde{R}_i(q_1,\dots,q_{n-1}). \label{eq:c-angmom-red}
\end{eqnarray}

\subsection{The reduced system \RS}\label{sec:red-sys}

The system of equations (\ref{eq:cc-kart}) is degenerate due to the presence of the $SO(d)$-symmetry.
In this section we  remove this symmetry and introduce the reduced system of equations.
Let us set $d=2$.
We  use the notation  $\tilde{R}_i=(\tilde{R}_{i,x},\tilde{R}_{i,y})$.
Let us fix $k_1\in \{1,\dots,n-1\}$, and consider the following set of equations
\begin{subequations}
\begin{align}
  q_i &= \frac{1}{m_i}f_i(q_1,\dots,q_n(q_1,\dots,q_{n-1})), \quad i\in \{1,\dots,n-1\}, i \neq k_1\label{eq:cc-red-i} \index{$q_i$}\\
  x_{k_1} &=  \frac{1}{m_{k_1}}f_{k_1,x}(q_1,\dots,q_n(q_1,\dots,q_{n-1})), \label{eq:cc-red-xk1}
\end{align}\label{eqn:cc-ref}
\end{subequations}
where $f_i = (f_{i,x}, f_{i, y})$.\index{$f_{k,x}$}
In the sequel, we use the abbreviation \RS to denote the reduced system (\ref{eqn:cc-ref}).
\RS  coincides with~(\ref{eqn:reduced}) with the
equation for $y_{k_1}$ dropped.  Observe that \RS has $2(n-1)-1$ equations for $q_1,\dots,q_{n-1} \in \mathbb{R}^2$. To obtain the same number
of variables we set $q_{k_1,y}=0$.
\RS no longer has $O(2)$ as a symmetry group, but still it is symmetric with respect to the reflections against
the coordinate planes.
Using the notation introduced in Section~\ref{subsubsec:com-red} \RS can be written as
\begin{subequations}
\begin{align}
  \tilde{R}_i(q_1,\dots,q_{n-1}) &=0, \quad i\in \{1,\dots,n-1\}, i \neq k_1 \label{eq:cc-R-red-i} \\
  \tilde{R}_{k_1,x}(q_1,\dots,q_{n-1}) &=0, \label{eq:cc-R-red-xk1},
\end{align}\label{eq:cc-R-red-xy}
\end{subequations}
with the requirement that $y_{k_1}=0$.
The next theorem addresses the question: whether from \RS we
obtain the solution of (\ref{eq:cc-kart})?

\begin{theorem}\label{thm:red-to-full}
   If $q = (q_1,\dots,q_n)$ is a solution of \RS and
    \begin{equation}\label{eqn:A1}
 x_{k_1} \neq x_n,
    \end{equation}
     then
    it is a normalized central configuration, i.e.\ it satisfies (\ref{eq:cc-kart}).
\end{theorem}
\proof
Without any loss of the generality we can assume that $k_1=n-1$.
We need to show that $\tilde{R}_{n-1,y}(q)=0$.
From (\ref{eq:c-angmom-red}) and~(\ref{eq:cc-R-red-xy}) we have
\begin{eqnarray*}
0 & = &\sum_{i=1}^{n-1}(q_i-q_n) \wedge \tilde{R}_i(q) =  (q_{n-1}-q_n)\wedge \tilde{R}_{n-1}(q)\\
  & = & (x_{n-1}-x_n) \tilde{R}_{n-1,y}(q)
\end{eqnarray*}
and our assertion follows immediately.
\qed

Observe that in the proof we did not use the assumption that $y_{k_1}=0$.

\subsection{The non-degenerate CC and the reduced system}

\begin{theorem}
\label{thm:non-deg-RS}
 Assume that $q=(q_1,\ldots,q_n)$  is a non-degenerate normalized CC. Then in a suitable coordinate system and after  some permutation of bodies   $q$ is a non-degenerate solution of \RS.
\end{theorem}
\proof
From the non-degeneracy assumption it follows that
\begin{equation*}
\rank(DF(q))=2n-1.
\end{equation*}
From this and Lemma~\ref{lem:rank-com-red} we obtain
\begin{equation*}
\rank\left(D\! \tilde{R}(q_1,\dots,q_{n-1})\right) = \rank\left(D\! F_{\mathrm{red}}(q_1,\dots,q_{n-1})\right)=2(n-1) -1.
\end{equation*}
Assume that we have $q_{n-1}=(x_{n-1},0)$, $x_{n-1}\neq x_n$ (see condition~(\ref{eqn:A1}) in Theorem~\ref{thm:red-to-full}), $x_{n-1} \neq 0$ and consider \RS
on the subspace $q_1,\dots,q_{n-2} \in \mathbb{R}^2$ and $x_{n-1} \in \mathbb{R}$. Now \RS has the same number of equations and variables equal to $2(n-1)-1$. We want to show that  the Jacobian matrix of this reduced system has the rank $2(n-1)-1$, which implies that $q$ is a non-degenerate
solution of \RS.
We have that $ D\! \tilde{R}(q_1,\dots,q_{n-1})$ is
\begin{eqnarray*}
\begin{bmatrix}
                                        D\tilde{R}_{1,x} \\[2ex]
                                        D\tilde{R}_{1,y} \\[2ex]
                                         D\tilde{R}_{2,x} \\[2ex]
                                        D\tilde{R}_{2,y} \\[2ex]
                                        \dots \\
                                         D\tilde{R}_{n-1,x} \\[2ex]
                                        D\tilde{R}_{n-1,y} \\
                                      \end{bmatrix}
                                      &  = &
                                      \begin{bmatrix}
                                      \frac{\partial \tilde{R}_{1,x}}{\partial x_1}, & \frac{\partial \tilde{R}_{1,x}}{\partial y_1}, & \frac{\partial \tilde{R}_{1,x}}{\partial x_2}, & \frac{\partial \tilde{R}_{1,x}}{\partial y_2}, & \dots, & \frac{\partial \tilde{R}_{1,x}}{\partial x_{n-1}}, & \frac{\partial \tilde{R}_{1,x}}{\partial y_{n-1}} \\[2ex]
                                      \frac{\partial \tilde{R}_{1,y}}{\partial x_1}, & \frac{\partial \tilde{R}_{1,y}}{\partial y_1}, & \frac{\partial \tilde{R}_{1,y}}{\partial x_2}, & \frac{\partial \tilde{R}_{1,y}}{\partial y_2}, & \dots, & \frac{\partial \tilde{R}_{1,y}}{\partial x_{n-1}}, & \frac{\partial \tilde{R}_{1,y}}{\partial y_{n-1}} \\[2ex]
                                      \frac{\partial \tilde{R}_{2,x}}{\partial x_1}, & \frac{\partial \tilde{R}_{2,x}}{\partial y_1}, & \ldots, & \ldots, & \dots,&  \frac{\partial \tilde{R}_{2,x}}{\partial x_{n-1}}, & \frac{\partial \tilde{R}_{2,x}}{\partial y_{n-1}} \\[2ex]
                                      \frac{\partial \tilde{R}_{2,y}}{\partial x_1}, & \frac{\partial \tilde{R}_{2,y}}{\partial y_1}, & \ldots, & \ldots, & \ldots, & \frac{\partial \tilde{R}_{2,y}}{\partial x_{n-1}}, & \frac{\partial \tilde{R}_{2,y}}{\partial y_{n-1}} \\[2ex]
                                      ...\\[2ex]
                                      \frac{\partial \tilde{R}_{n-1,y}}{\partial x_1}, & \frac{\partial \tilde{R}_{n-1,y}}{\partial y_1}, & \ldots, & \ldots, & \dots, & \frac{\partial \tilde{R}_{n-1,y}}{\partial x_{n-1}}, & \frac{\partial \tilde{R}_{n-1,y}}{\partial y_{n-1}}
                                      \end{bmatrix}\\[1ex]
& = & \begin{bmatrix}
\mbox{}\quad  \frac{\partial \tilde{R}}{\partial x_1}, \mbox{}\quad\quad &   \frac{\partial \tilde{R}}{\partial y_1},
                                                         & \mbox{}\quad\quad  \dots,
                                                         & \mbox{}\quad \dots
                                                         & \mbox{}\quad \dots
                                                         &  \mbox{}\quad \frac{\partial \tilde{R}}{\partial x_{n-1}},
                                                         & \mbox{}\quad  \frac{\partial \tilde{R}}{\partial y_{n-1}}
                                                    \end{bmatrix}
\end{eqnarray*}
Observe that the jacobian matrix of \RS is equal $D\! \tilde{R}(q_1,\dots,q_{n-1})$ with removed the last column (which is consequence of the restriction to $y_{n-1}=0$) and the last row (which is consequence of dropping the equation $\tilde{R}_{n-1,y}=0$). We need to show that such removal does not change the rank of the matrix.
From  (\ref{eq:c-angmom-red}) we have
\begin{equation}
 0=\sum_{i=1}^{n-1}(q_i-q_n) \wedge \tilde{R}_i(q_1,\dots,q_{n-1}).
\end{equation}
By taking partial derivatives of the above equation with respect to $x_j$ and $y_j$ and evaluating at $q$ (we have $\tilde{R}_i(q)=0$) we obtain
\begin{equation}
  0=\sum_{i=1}^{n-1}(q_i-q_n) \wedge D \tilde{R}_i(q_1,\dots,q_{n-1})= \sum_{i=1}^{n-1}\left((x_i-x_n)  D \tilde{R}_{i,y}- (y_i-y_n)  D \tilde{R}_{i,x} \right)
\end{equation}
The above equation is the linear combination of the rows in the matrix $D \tilde{R}$. If $x_{n-1}-x_n \neq 0$, then   the last row can be expressed as the linear combination of other rows, hence it can be removed from the matrix without changing its rank.

Now we show that last row $\frac{\partial \tilde{R}}{\partial y_{n-1}}$ is a linear combination of other rows.
Let $O(t)$ be the rotation by angle $t$.  It acts on configuration $q$ as follows $q_i(t) = O(t)q_i$ and $q(t)=(q_1(t),\dots,q_{n-1}(t),q_n(t))$
is a normalized central configuration if $q$ is.   Observe that $\frac{d}{dt}q_i(t)_{t=0}=(-y_i,x_i)$.
From the rotational symmetry we have
\begin{equation}
  \tilde{R}(O(t)q)=0
\end{equation}
By  taking the derivative with respect to the angle  for $t=0$ we obtain
\begin{equation}
  0=\sum_{i=1}^{n-1}\left( -\frac{\partial \tilde R}{\partial x_i}y_i + \frac{\partial \tilde R}{\partial y_i}x_i \right)
\end{equation}
If $x_{n-1} \neq 0$, then we can express  $\frac{\partial \tilde R}{\partial y_{n-1}}$ (the last column in the matrix $D \tilde{R}$)
in terms of other columns.
Therefore we can remove the last row and the last column from the matrix $D \tilde{R}(q)$ without decreasing its rank. Therefore the rank
of \RS is $2(n-1)-1$, which shows the non-degeneracy of $q$ as the solution of \RS.

Now the question remains: whether we can always achieve that $x_{n-1} \neq x_n$ and $x_{n-1} \neq 0$. First we take any $q_{i_0}\neq 0$ (there can be only one body at the origin) and we chose coordinate frame so that $q_{i_0}=(x_{i_0},0)$. Then we look for $j \neq i_0$ such that $x_{j} \neq x_{i_0}$.  Observe that due to (\ref{eq:cc-cofmass-l}) such $j$ always exists.  Now we change the numeration of bodies so that $i_0 \to n-1$ and $j \to n$.
\qed

\begin{rem}
From the proof of the above theorem it follows that a normalized CC, such that $x_{n-1} = x_n$ and $y_{n-1}=0$ is always a degenerate solution
of \RS.
\end{rem}

\begin{theorem}
\label{thm:RS-non-deg-CC}
  Assume that $q=(q_1,\dots,q_n)$  is a non-degenerate solution of \RS, such that $x_{n-1} \neq x_n $. Then $q$ is a non-degenerate normalized central configuration.
\end{theorem}
An easy proof is left to reader.

\section{CCs for $(N+k)$-bodies}
\label{sec:cc-N+k}

\subsection{Some a-priori bounds for normalized CCs}
\label{subsec:apriori-bnds}

Let us recall two theorems from \cite{MZ}: first states an upper bound on the size of a normalized CC; second states the lower bound for distances between bodies in normalized CC. We use them repeatedly in the sequel.

\begin{theorem}\cite[Thm. 11]{MZ}
\label{thm:ncc-upperBound}
Given positive masses $m_i$, $i=1,\dots,n$. Let $M=\sum_i m_i$.
Assume  $q = (q_1, \ldots, q_n)$ is a normalized $\cc{n}$. Then
\renewcommand{\arraystretch}{1.6}
\begin{equation}
  \max_{i} |q_i| \leqslant M^{1/3} \left\{
                                   \begin{array}{ll}
                                     n-1, & \hbox{$n\geqslant  2$;} \\
                                     \left(2^{1/3}+2^{-2/3} \right) (n-2)^{2/3}, & \hbox{$n\geqslant  4$.}
                                   \end{array}
                                 \right.
\end{equation}
\renewcommand{\arraystretch}{1.2}
\end{theorem}
In \cite{MZ} this theorem was stated for $M=1$, but a simple scaling argument gives factor $M^{1/3}$.

\begin{theorem}\cite[Thm. 6]{MZ}
  \label{thm:ncc-lbnd}
  Given positive masses $m_i$, $i=1,\dots,n$. Let $M=\sum_i m_i$.
  Assume that a normalized $\cc{n}$ satisfies $|q_i| \leqslant  R$ for $i=1,\dots,n$. Then
\begin{equation}
   r_{ij} > \frac{m_i m_j}{M R^2}, \quad 1\leqslant  i < j \leqslant  n. \label{eq:lb-dist}
\end{equation}
\end{theorem}

\subsection{Study of the limit of small masses going to zero}
\label{subsec:msm}

To study central configurations with several masses going to zero  we need the notion of the restricted central configuration problem.

\begin{definition}
\label{def:rel-equi-restricted}
Let $\qnm=(q_1,q_2,\dots,q_N) \in \left(\mathbb{R}^2\right)^N$ with masses $m = (m_i)_{i = 1, \ldots, N}$.  {\em A potential $\vres$ for the restricted $(N+1)$-body problem induced by configuration $\qnm$} for $p \in \mathbb{R}^2$ is defined by
\begin{equation}
\vres(\qnm, p) = \frac{p^2}{2} + \sum_{i=1}^N \frac{m_i}{|p - q_i|},   \label{eq:Vres}
\end{equation}

Let $x^*$ be a solution of
\begin{equation}\label{eq:rel-equi-restricted}
\frac{\partial \vres}{\partial p}(\qnm,x^*)=0.
\end{equation}
Then $x^*$ is called a {\em relative normalized CC in the restricted $(N+1)$-body problem} for the configuration $\qnm$ --- in short $x^*\in \re{\qnm}$.

The solution $x^*$ is called {\em non-degenerate}, if $\frac{\partial^2 \vres}{\partial p^2} (\qnm,x^*)$ is an isomorphism\footnote{Everywhere we write $\frac{\partial \vres}{\partial q_i}(\qnm, y)$ or $\frac{\partial \vres}{\partial p}(\qnm, x)$, we have in mind  that the potential $\vres(\qnm, p)$ is a function of $N+1$ variables $(q_1, \ldots, q_n, p)$ and  it means the derivative with respect to variables $q_i$ or $p$, respectively, at the point $(\qnm, x)$.}.
\end{definition}

We are interested in the $(N+k)$-body problem, where there is  $N$ heavy  and $k$ light bodies.
 The meaning of the next theorem (Theorem~\ref{thm:limits}) is as follows:
in the central configuration the light bodies gather in several clusters around
$x^*\in \re{Q_N}$, i.e.\ solutions of~(\ref{eq:rel-equi-restricted})  (see Fig.~\ref{fig:cluster-ex}).

 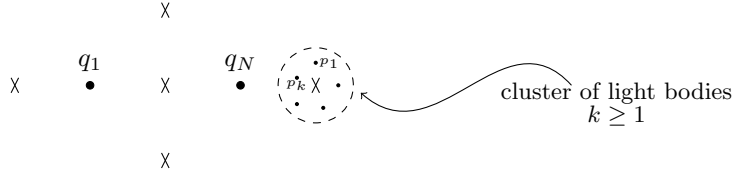
\begin{figure}[H]
  \centering
\begin{tikzpicture}
\node[] at (0.0, 1.3) {$q_1$};
\draw[fill] (0.0, 1.0) circle (0.05cm);

\draw[fill] (-0.95, 0.9) -- (-1.05, 1.1);
\draw[fill] (-0.95, 1.1) -- (-1.05, 0.9);
\draw[fill] (0.95, -0.1) -- (1.05, 0.1);
\draw[fill] (0.95, 0.1) -- (1.05, -0.1);
\draw[fill] (0.95, 0.9) -- (1.05, 1.1);
\draw[fill] (0.95, 1.1) -- (1.05, 0.9);
\draw[fill] (0.95, 1.9) -- (1.05, 2.1);
\draw[fill] (0.95, 2.1) -- (1.05, 1.9);
\draw[fill] (2.95, 0.9) -- (3.05, 1.1);
\draw[fill] (2.95, 1.1) -- (3.05, 0.9);

\node[] at (2.0, 1.3) {$q_N$};
\draw[fill] (2.0, 1.0) circle (0.05cm);

\draw[dashed] (3.0, 1.0) circle (5mm);

\draw[fill] (3.0, 1.3) circle (0.2mm);
\node[] at (3.18, 1.3) {\tiny $p_1$};

\draw[fill] (3.1, 0.7) circle (0.2mm);

\draw[fill] (2.75, 1.1) circle (0.2mm);
\node[] at (2.75, 1.0) {\tiny $p_k$};

\draw[fill] (3.3, 1.0) circle (0.2mm);

\draw[fill] (2.75, 0.75) circle (0.2mm);

\draw[->] (6.4, 1.0) .. controls (5.5, 2.0) and (4.6, 0.0) .. (3.6, 0.9);
\node[] at (7.0, 0.9) {\small cluster of light bodies};
\node[] at (7.0, 0.6) {\small $k\ge 1$};

\end{tikzpicture}
\captionof{figure}{The crosses indicate the possible positions of $x^*\in \re{\qnm}$.}\label{fig:cluster-ex}
\end{figure}

We use a parameter $\eps$ to describe the behaviour of masses and  positions with $\eps$  tending to zero .
 For the masses, we assume that $\mu_j \to 0$ for $j=1,\dots,k$ and the example we have in mind is $\mu_j=\eps \tilde{\mu}_j$ with $\eps \to 0$ and $\tilde{\mu}_j$ being some positive constants. For $m_i(\eps)$ we assume that they converge to some non-zero $\bbar{m}_i$.

For the positions, we consider a normalized central configuration $(Q_N(\eps),P_k(\eps))$ for the $(N+k)$-bodies with masses $m_i(\eps)$ and $\mu_j(\eps)$. From Theorem~\ref{thm:ncc-upperBound} it follows that we can find a subsequence of $\{\eps_s\}_{s \in \mathbb{N}}$ such that the $(Q_N(\eps_s),P_k(\eps_s))$ are converging to some limit. We want to investigate the nature of this limit.

\begin{theorem}
\label{thm:limits}
Let $(m_1(\eps),\dots,m_N(\eps),\mu_1(\eps),\dots,\mu_k(\eps))$ are positive masses depending on parameter $\eps$ (which might be discrete or continuous) and a sequence of normalized CC for these masses $(q_1(\eps),\dots,q_N(\eps),p_1(\eps),\dots,p_k(\eps))$, such that for $\eps \to 0$ holds
\begin{eqnarray}
  m_i(\eps) &\to& \bbar{m}_i >0, \quad i=1,\dots,N \\
  \mu_i(\eps) &\to& 0, \quad i=1,\dots,k, \\
  q_i(\eps) &\to& \bbar{q}_i, \quad p_j(\eps) \to \bbar{p}_j, \qquad i=1,\dots,N \quad j=1,\dots,k.
\end{eqnarray}
Then
\begin{enumerate}[label=(\Alph*)]
  \item $\bbar{Q}_N=(\bbar{q}_1,\dots,\bbar{q}_N)$ is a normalized CC for masses $(\bbar{m}_1,\dots,\bbar{m}_N)$,
  \item $\bbar{p}_j\in \re{\bbar{Q}_N}$ for $j = 1, \ldots, k$.
\end{enumerate}
\end{theorem}

\proof
First define {\em a pleiad  $\C$} (of light bodies) --- a set of indices of light bodies such that all bodies in the pleiad have the same limit point when $\eps$ tends to 0.

Note that a pleiad  may contain one or several light bodies.
For the pleiad  $\C$ by $\bbar{p}(\C)$ we denote the corresponding limit point.
Let us divide all light bodies into $\ncl$ pleiades $\C_i$, with $i=1,\dots,\ncl$.

Since limit points of different pleiades are different and there is a finite number $\ncl$ of these points, thus we can take $\delta_\C >0$ such that  holds
\begin{equation}
  |\bbar{p}(\C_s) - \bbar{p}(\C_t) | > \delta_\C  \label{eq:deltaC}
\end{equation}
for all $s, t=1,\dots,\ncl$, $s\neq t$.
Let us set
\begin{eqnarray}
  M(\eps)&=&\sum_{i=1}^N m_i(\eps) + \sum_{j=1}^k \mu_j(\eps), \\
  \overline{M}&=&\lim_{\eps \to 0} M(\eps)= \sum_{i=1}^N \lim_{\eps \to 0} m_i(\eps) = \sum_{i=1}^N \bbar{m}_i.
\end{eqnarray}

Justification of the below facts constitutes the main  part of the proof:
\begin{lemma}\label{lm:limit-heavy}
Under the assumptions of Theorem~\ref{thm:limits}, the limit configuration of heavy bodies does not contain any collision (i.e.\ $\exists \RH>0 \colon   |\bbar{q}_i - \bbar{q}_j |  \geq \RH, \quad  |q_i(\eps) - q_j(\eps) |  \geq \RH$ for  $i,j=1,\dots,N, \ i \neq j,\ \eps \leq \eps_0$).
\end{lemma}
\begin{lemma}\label{lm:limit-light-heavy}
Under the assumptions of Theorem~\ref{thm:limits}, the light bodies cannot converge to any of the heavy bodies,(i.e.\ $\exists r > 0\ \forall i, j\colon |\bbar{q_i} - \bbar{p_j}|\ge r > 0$).
\end{lemma}

\proof (of Lemma~\ref{lm:limit-heavy})
First we show the existence of a lower bound for $r_{ij}(\eps)=|q_i(\eps) - q_j(\eps)|$,  $i,j=1,\dots,N$,  i.e.\  a lower bound for mutual distances between the heavy bodies.

   From Theorem~\ref{thm:ncc-lbnd} we have that the distance $r_{ij}(\eps)$ between any pair of  heavy bodies in CC has lower bound
\begin{equation*}
  r_{ij}(\eps) > \frac{m_i(\eps) m_j(\eps)}{M(\eps) R(\eps)^2},  
\end{equation*}
where  $R(\eps)=\max\left\{\max_{i=1,\dots,N} |q_i(\eps)|, \max_{j=1,\dots,k} |p_j(\eps)| \right\}$. From Theorem~\ref{thm:ncc-upperBound}  we know an upper bound on $R(\eps)$
 \begin{equation}
  R(\eps) \leqslant M(\eps)^{1/3}  \left\{
                                   \begin{array}{ll}
                                     N+k-1, & \hbox{$N+k\geqslant  2$;} \\
                                     \left(2^{1/3}+2^{-2/3} \right) (N+k-2)^{2/3}, & \hbox{$N+k\geqslant  4$}
                                   \end{array}
                                 \right.  \label{eq:R-eps-upbnd}
\end{equation}
hence
\begin{equation}
  r_{ij}(\eps) > \frac{m_i(\eps) m_j(\eps)}{M(\eps)^{5/3}(N+k-1)^2},  \label{eq:rij-eps-lbnd}
\end{equation}

Passing to the limit  $\eps \to 0$  in (\ref{eq:rij-eps-lbnd}) we obtain
\begin{equation}
  |\bbar{q}_i - \bbar{q}_j | \geq  \bbar{m}_i \bbar{m}_j\left(\bbar{M}^{5/3}(N+k-1)^2\right)^{-1}, \quad i,j=1,\dots,N, \ i\neq j.  \label{eq:lBndHb}
\end{equation}

This establishes the existence of $\RH > 0$ (subscript H is for heavy bodies ) and $\eps_0 >0$, such that
\begin{equation}
   |\bbar{q}_i - \bbar{q}_j |  \geq \RH, \quad  |q_i(\eps) - q_j(\eps) |  \geq \RH, \qquad  i,j=1,\dots,N, \ i \neq j,\ \eps \leq \eps_0. \label{eq:dist-h-h}
\end{equation}

This completes the proof of Lemma~\ref{lm:limit-heavy}.

\proof (of Lemma~\ref{lm:limit-light-heavy})
We are going to show that the light bodies cannot converge to any of the heavy bodies. That means  that
there exists $r>0$, such that for sufficiently small $\eps$
\begin{equation*}
  |q_i(\eps) - p_j(\eps)| \geq r, \qquad i=1,\dots,N, \quad j=1,\dots,\ncl.
\end{equation*}
We know that $\lim_{\eps\to 0}q_i(\eps) = \bbar{q}_i$ and $\lim_{\eps\to 0}p_j(\eps) = \bbar{p}(\C_j)$, thus in fact we show that there exists $r>0$, such that
\begin{equation}
  |\bbar{q}_i - \bbar{p}(\C_j)| \geq r, \qquad i=1,\dots,N, \quad j=1,\dots,\ncl.  \label{eq:lim-dist-h-l}
\end{equation}

Since all light bodies in a pleiad converge to the same limit point, from~(\ref{eq:deltaC}) it follows that there exists
$\eps_1>0$  such  that for all $s, t=1,\dots,\ncl$, $s \neq t$ holds
\begin{equation}
  |p_j(\eps) - p_l(\eps) | \geq \delta_\C,   \quad j \in \C_s , l \in \C_t, \quad \eps < \eps_1,  \label{eq:dist-l-l}
\end{equation}
i.e.\ we have a uniform positive lower bound for the distance between the light bodies in different pleiades.

By contradiction, assume that for some $i_0$ and some pleiad $\C_s$ we have
\begin{equation}\label{eqn:contr}
  |\bbar{q}_{i_0}- \bbar{p}(\C_s)| = 0.
\end{equation}
Let us take
\begin{equation}
 r < \min\left\{\RH\sqrt{m_{\min}\left(\RH^2\bbar{M}^{2/3}(N+k-1) +   4\bbar{M}\right)^{-1}}, \frac{1}{2}\RH\right\},\label{eq:r-estm}
 \end{equation}
 where $m_{\min} = \min\{m_{i}: i = 1, \ldots, N\}$.

For sufficiently small $\eps$ holds
\begin{equation}
  |q_{i_0}(\eps) - p_j(\eps)| \leq r, \quad \forall j \in \C_s.  \label{eq:dist-qi0-clu}
\end{equation}

Let us fix  $\eps < \min\{\eps_0, \eps_1\}$ satisfying~(\ref{eq:dist-qi0-clu}). Let $j_0 \in \C_s$ be such that
\begin{equation}
  |p_{j_0}(\eps)-q_{i_0}(\eps)| \geq |p_j(\eps)-q_{i_0}(\eps)|, \quad j\in \C_s
\end{equation}
i.e.\ $p_{j_0}(\eps)$ is the point in the pleiad $\C_s$ which is  furthest away from $q_{i_0}(\eps)$.

Consider now the direction connecting $p_{j_0}(\eps)$ with $q_{i_0}(\eps)$. Without loss of the generality, we can assume that  the direction is the $x$-coordinate direction, i.e.\ $y(p_{j_0}(\eps))=y(q_{i_0}(\eps))$, where by $x(p)$, $y(p)$ denote $x$ and $y$ components of $p\in\reals^2$, respectively.
In such situation the $x$-component of forces acting on $p_{j_0}(\eps)$ coming from the heavy body $i_0$ and all other light bodies in the pleiad $\C_s$ have the same sign (see Fig.~\ref{fig:pleiad}).

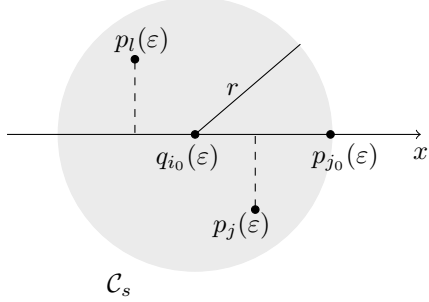
\begin{figure}[H]
 \centering
  \begin{tikzpicture}

\draw[fill, lightGrey] (1.0, 2.0) circle (1.82cm);
\node[] at (0.0, 0.0) {$\mathcal{C}_s$};

\draw[->, thin] (-1.5, 2.0) -- (4.0, 2.0);
\node[] at (4.0, 1.75) {$x$};

\draw[fill] (1.0, 2.0) circle (0.05cm);
\node[] at (0.9, 1.7) {$q_{i_0}(\eps)$};

\draw[thin] (1.0, 2.0) -- (2.4, 3.2);
\node[] at (1.5, 2.6) {$r$};

\draw[fill] (2.8, 2.0) circle (0.05cm);
\node[] at (3.0, 1.7) {$p_{j_0}(\eps)$};

\draw[fill] (1.8, 1.0) circle (0.05cm);
\node[] at (1.62, 0.78) {$p_j(\eps)$};
\draw[dashed] (1.8, 1.03) -- (1.8, 2.0);

\draw[fill] (0.2, 3.0) circle (0.05cm);
\node[] at (0.3, 3.25) {$p_l(\eps)$};
\draw[dashed] (0.2, 2.97) -- (0.2, 2.0);

\end{tikzpicture}
\caption{The $x$-component of forces acting on $p_{j_0}(\eps)$ coming from the heavy body $i_0$ and all other light bodies $p$ in the pleiad $\mathcal{C}_s$ have the same sign; pleiades of light bodies are symbolically visualized by grey circles.}\label{fig:pleiad}
\end{figure}

Let us rewrite (\ref{eq:cc}) with $\lambda=1$ for  $x$-component of the equation corresponding $p_{j_0}$-th body. We  group terms to highlight which terms are dominating and we take an absolute value of both sides:

\begin{eqnarray}
   &  & \left|\frac{m_{i_0}(\eps)(x(p_{j_0}(\eps))-x(q_{i_0}(\eps)))}{|p_{j_0}(\eps)-q_{i_0}(\eps)|^3} + \sum_{j \in \C_s, j \neq j_0}  \frac{\mu_{j}(\eps)(x(p_{j_0}(\eps))-x(p_j(\eps)))}{|p_{j_0}(\eps)-p_j(\eps)|^3} \right|  \nonumber \\
   & = &\left|\sum_{i \neq i_0,i=1}^N \frac{m_{i}(\eps)(x(p_{j_0})-x(q_{i}(\eps)))}{|p_{j_0}(\eps)-q_{i}(\eps)|^3} + \sum_{j \notin \C_s}  \frac{\mu_{j}(\eps)(x(p_{j_0})-x(p_j(\eps)))}{|p_{j_0}(\eps)-p_j(\eps)|^3} - x(p_{j_0}(\eps))\right|  \label{eq:ncc-px-comp}
\end{eqnarray}
We show that, if $r$ is small enough,  lhs of~(\ref{eq:ncc-px-comp}) dominates rhs of~(\ref{eq:ncc-px-comp}).
First, observe that the inequalities~(\ref{eqn:a})--(\ref{eqn:c}) are true (we precede each equation with a short proof).

Since our choice of $j_0$ is such that the expressions $x(p_{j_0}(\eps))-x(q_{i_0}(\eps))$  and $x(p_{j_0}(\eps))-x(p_{j}(\eps))$ for $j \in \C_s$ have the same sign (see Fig.~\ref{fig:pleiad}) we obtain
\begin{eqnarray*}
 {\rm lhs(\ref{eq:ncc-px-comp})} & \geq  & \left|\frac{m_{i_0}(\eps)(x(p_{j_0}(\eps))-x(q_{i_0}(\eps)))}{|p_{j_0}(\eps)-q_{i_0}(\eps)|^3} \right| = \frac{m_{i_0}(\eps)}{|p_{j_0}(\eps)-q_{i_0}(\eps)|^2} \geq  \frac{m_{i_0}(\eps)}{r^2}
\end{eqnarray*}
hence
\begin{equation} \label{eqn:a}
\left| \frac{m_{i_0}(\eps)(x(p_{j_0}(\eps))-x(q_{i_0}(\eps)))}{|p_{j_0}(\eps)-q_{i_0}(\eps)|^3} + \sum_{j \in \C_s, j \neq j_0}  \frac{\mu_{j}(\eps)(x(p_{j_0}(\eps))-x(p_j(\eps)))}{|p_{j_0}(\eps)-p_j(\eps)|^3} \right| \geq \frac{m_{i_0}(\eps)}{r^2}
\end{equation}

By Lemma~\ref{lm:limit-heavy}, conditions (\ref{eq:dist-qi0-clu}), (\ref{eq:r-estm})) and the triangle inequality there is
\begin{eqnarray*}
  \left|  \sum_{i \neq i_0,i=1}^N \frac{m_{i}(\eps)(x(p_{j_0})-x(q_{i}(\eps)))}{|p_{j_0}(\eps)-q_{i}(\eps)|^3}\right|
   & \leq &
    \sum_{i \neq i_0,i=1}^N \frac{m_{i}(\eps)}{|p_{j_0}(\eps)-q_{i}(\eps)|^2}   \\
 & \leq  &
    \sum_{i \neq i_0,i=1}^N \frac{m_{i}(\eps)}{\left(|q_{i_0}(\eps)-q_{i}(\eps)| - |q_{i_0}(\eps)-p_{j_0}(\eps)| \right)^2} \\
    &  \leq &  \sum_{i \neq i_0,i=1}^N \frac{m_{i}(\eps)}{(\RH -  r)^2}
     \leq    \left(\sum_{i \neq i_0,i=1}^N m_{i}(\eps)\right) \frac{1}{(\RH/2)^2}\\
   & \leq &  \frac{4 M(\eps)}{\RH^2}
\end{eqnarray*}
hence
\begin{equation}\label{eqn:b}
\left|  \sum_{i \neq i_0,i=1}^N \frac{m_{i}(\eps)(x(p_{j_0})-x(q_{i}(\eps)))}{|p_{j_0}(\eps)-q_{i}(\eps)|^3}\right| \leq   \frac{4 M(\eps)}{\RH^2}
\end{equation}

By~(\ref{eq:dist-l-l}) we obtain\\
\begin{eqnarray*}
 \left|\sum_{j \notin \C_s}  \frac{\mu_{j}(\eps)(x(p_{j_0})-x(p_j(\eps)))}{|p_{j_0}(\eps)-p_j(\eps)|^3} \right|
 \leq \sum_{j \notin \C_s}  \frac{\mu_{j}(\eps)}{|p_{j_0}(\eps)-p_j(\eps)|^2} \leq \sum_{j \notin \C_s}  \frac{\mu_{j}(\eps)}{\delta_\C^2}
 \leq \left(\sum_{j \notin \C_s}  \mu_{j}(\eps)\right)\frac{1}{\delta_\C^2}
\end{eqnarray*}
hence
\begin{equation}\label{eqn:c}
 \left|\sum_{j \notin \C_s}  \frac{\mu_{j}(\eps)(x(p_{j_0})-x(p_j(\eps)))}{|p_{j_0}(\eps)-p_j(\eps)|^3} \right|
 \leq \left(\sum_{j \notin \C_s}  \mu_{j}(\eps)\right)\frac{1}{\delta_\C^2}
\end{equation}

Finally, from Theorem~\ref{thm:ncc-upperBound} we have
\begin{equation}\label{eqn:d}
  | x(p_{j_0}(\eps))| \leq  M^{1/3}(\eps) R(\eps)
\end{equation}

Notice that by (\ref{eqn:a})--(\ref{eqn:d}) we have
\begin{eqnarray*}
\frac{m_{i_0}(\eps)}{r^2}\leq \mbox{lhs(\ref{eq:ncc-px-comp}) = rhs(\ref{eq:ncc-px-comp})} & \leq &  M^{1/3}(\eps) R(\eps) +   \frac{4 M(\eps)}{\RH^2} + \left(\sum_{j \notin \C_s}  \mu_{j}(\eps)\right)\frac{1}{\delta_\C^2}
\end{eqnarray*}

hence
\begin{eqnarray*}
r^2 & \geq &  m_{i_0}(\eps)\left(M^{1/3}(\eps) R(\eps) +   \frac{4 M(\eps)}{\RH^2} + \left(\sum_{j \notin \C_s}  \mu_{j}(\eps)\right)\frac{1}{\delta_\C^2}\right)^{-1}.
\end{eqnarray*}

Passing to the limit $\eps \to 0$ we obtain
\begin{eqnarray*}
r^2 & \geq &  m_{\min}\left(\bbar{M}^{2/3}(N+k-1) +   \frac{4\bbar{M}}{\RH^2}\right)^{-1}\\
  & = &   m_{\min}\RH^2\left(\RH^2\bbar{M}^{2/3}(N+k-1) +   4\bbar{M}\right)^{-1}.
\end{eqnarray*}
Hence
\begin{equation}
r \ge \RH\sqrt{m_{\min}\left(\RH^2\bbar{M}^{2/3}(N+k-1) +   4\bbar{M}\right)^{-1}}.\label{eqn:contradiction}
\end{equation}
However, this contradicts~(\ref{eq:r-estm}), completing the proof of Lemma~\ref{lm:limit-light-heavy}.

(A)\
Now we prove that $(\bbar{q}_1,\dots,\bbar{q}_N)$ is a normalized CC for the $N$-body problem.

Let us rewrite (\ref{eq:cc}) for $i=1,\dots,N$ in the following way:
\begin{equation*}
  q_i(\eps)=\sum_{j=1, j \neq i}^N \frac{m_j(\eps)(q_i(\eps) - q_j(\eps))}{|q_i(\eps) - q_j(\eps)|^3} + \sum_{j=1}^k \frac{\mu_j(\eps)(q_i(\eps) - p_j(\eps))}{|q_i(\eps) - p_j(\eps)|^3}.
\end{equation*}

Lemmas~\ref{lm:limit-heavy} and~\ref{lm:limit-light-heavy} imply  that when $\eps$ tends to 0, denominators in the above equations are bounded from below by some positive
numbers. Hence we obtain

 \begin{equation*}
  \bbar{q}_i=\sum_{j=1, j \neq i}^N \frac{\bbar{m}_j(\bbar{q}_i - \bbar{q}_j)}{|\bbar{q}_i - \bbar{q}_j|^3},
\end{equation*}
showing that $\bbar{Q}_N$ is $\cc{N}$ for masses $\bbar{m}_1,\dots,\bbar{m}_N$. This completes the proof of part (A).

(B)
It remains to prove that $\bbar{p}_j\in \re{\bbar{Q}_N}$ for $j = 1, \ldots, k$.
Write equations  (\ref{eq:cc}) for the light bodies as follows
\begin{equation}
    \mu_j p_j(\eps)=\sum_{i=1}^N \frac{\mu_j m_i(\eps)(p_j(\eps) - q_i(\eps))}{|p_j(\eps) - q_i(\eps)|^3} + \sum_{i=1, i \neq j}^k \frac{\mu_i \mu_j(\eps)(p_j(\eps) - p_i(\eps))}{|q_j(\eps) - p_i(\eps)|^3}.  \label{eq:ncc-lb}
\end{equation}

Let us take a pleiad $\C_s$. Bodies of that pleiad converge to $p^*=\bbar{p}(\C_s)$, when $\eps$ tends to 0.
Notice that for  $\Theta(\eps) = \sum_{j \in \C_s} \mu_j(\eps)$
\begin{eqnarray*}
  c(\eps)&=&\frac{1}{\Theta(\eps)} \sum_{j \in \C_s} \mu_j(\eps) p_j(\eps)
\end{eqnarray*}
 is a center of mass of $\C_s$ and
$c(\eps) \to p^*.$
We add  equations (\ref{eq:ncc-lb}) for $j \in \C_s$  and divide by $\Theta(\eps)$ to obtain
\begin{eqnarray}
  c(\eps) & = & \frac{1}{\Theta(\eps)}\sum_{j\in \C_s} \sum_{i=1}^N \frac{\mu_j(\eps) m_i(\eps)(p_j(\eps) - q_i(\eps))}{|p_j(\eps) - q_i(\eps)|^3} + \frac{1}{\Theta(\eps)} \sum_{j\in \C_s} \sum_{\substack{i=1,\\ i \notin \C_s}}^k \frac{\mu_j(\eps) \mu_i(\eps)(p_j(\eps) - p_i(\eps))}{|p_i(\eps) - p_j(\eps)|^3}  \nonumber \\
     & = & \sum_{i=1}^N   m_i(\eps)\sum_{j\in \C_s}   \frac{\mu_j(\eps)}{\Theta(\eps)} \frac{p_j(\eps) - q_i(\eps)}{|p_j(\eps) - q_i(\eps)|^3}
       + \sum_{\substack{i=1,\\ i \notin \C_s}}^k \mu_i(\eps) \sum_{j\in \C_s} \frac{\mu_j(\eps)}{\Theta(\eps)}  \frac{ (p_j(\eps) - p_i(\eps))}{|p_i(\eps) - p_j(\eps)|^3} . \label{eq:cc-eq-cCc}
\end{eqnarray}

From Lemma~\ref{lm:limit-light-heavy} it follows that $p^* \neq \bbar{q}_i $ and
\begin{equation*}
   \frac{p_j(\eps) - q_i(\eps)}{|p_j(\eps) - q_i(\eps)|^3} \to \frac{p^* - \bbar{q}_i}{|p^* - \bbar{q}_i|^3},
\end{equation*}
and   since by (\ref{eq:deltaC}) $p^* \neq \bbar{p}_j$ for $j \notin \C_s$  we have
\begin{equation*}
  \frac{p_i(\eps) - p_j(\eps)}{|p_i(\eps) - p_j(\eps)|^3}  \to  \frac{ p^* - \bbar{p}_j}{|p^* - \bbar{p}_j|^3}, \quad i \in \C_s, j \notin \C_s
\end{equation*}

Taking this into account when passing to the limit $\eps \to 0$ in (\ref{eq:cc-eq-cCc}) we obtain
\begin{eqnarray*}
  p^* & = &
  \sum_{i=1}^N    \frac{\bbar{m}_i(p^* - \bbar{q}_i)}{|p^* - \bbar{q}_i|^3}.
\end{eqnarray*}
Therefore $p^*$ is  a relative normalized central configuration for the restricted $(N+1)$-problem with the configuration $\bbar{Q}_N$.

This completes the proof of part (B) and the proof of the whole theorem.
\qed

\subsection{The shape of the pleiads of light bodies}

To better understand the behavior of pleiad of light bodies in the limit of vanishing masses,
 we investigate in more details the potential $V$ given in~(\ref{eq:cc-pot}). First, we separate the interactions between bodies into three groups: heavy-heavy, heavy-light and light-light.
Thus we obtain
\begin{eqnarray*}
  V(Q_N,P_k)
  & = & \sum_{i=1}^N \frac{m_i q_i^2}{2} + \sum_{i = 1}^{N-1}\left( \sum_{j = i+ 1}^{N}\frac{m_i m_j}{|q_i - q_j|}\right)\\
  &  & +  \sum_{j=1}^{k} \frac{\mu_{j} p_{j}^2}{2} + \sum_{i = 1}^{N}\left( \sum_{j = 1}^{k}\frac{m_i \mu_j}{|q_i - p_j|}\right)\\
  &  &  +
    \sum_{i = 1}^{k-1}\left( \sum_{j = i+ 1}^{k}\frac{\mu_i \mu_j}{|p_i - p_j|}\right),
\end{eqnarray*}
where heavy bodies are $Q_N = (q_i)_{i = 1, \ldots, N}$ with masses $m = (m_i)_{i = 1, \ldots, N}$ (hereinafter referred to as $\qnm$) and light bodies are $P_k = (p_j)_{j = 1, \ldots, k}$ with masses $\mu = (\mu_j)_{j = 1, \ldots, k}$ (hereinafter referred to as $\pkm$).

Then (see Def.~\ref{def:rel-equi-restricted})
\begin{equation*}
  V(\qnm, \pkm)  =  V(\qnm) + \sum_{j=1}^k  \mu_j\vres(\qnm,p_j) +   \sum_{1\leq i<j\leq k} \frac{\mu_i \mu_j}{|p_i - p_j|}.
\end{equation*}
Let us denote
\begin{equation}
W(\qnm, \pkm) = \sum_{j=1}^k  \mu_j\vres(\qnm,p_j) +   \sum_{1\leq i<j\leq k} \frac{\mu_i \mu_j}{|p_i - p_j|}
\end{equation}
thus
\begin{equation*}
  V(\qnm, \pkm)  =  V(\qnm) + W(\qnm, \pkm) .
\end{equation*}

Now, equations~(\ref{eq:cc-DV}) for CC become
\begin{eqnarray}
\frac{\partial V}{\partial q_i} & = & \frac{\partial V}{\partial q_i}(\qnm) + \sum_{j=1}^k \mu_j \frac{\partial \vres}{\partial q_i}(\qnm, p_j)=0, \qquad \underline{i=1,\dots, N},  \label{eq:cc-big-masses} \\
\frac{\partial V}{\partial p_j} & = & \frac{\partial W}{\partial p_j}(\qnm, \pkm) = \mu_j \left(\frac{\partial \vres}{\partial p}(\qnm,p_j) + \sum_{\substack{1\leq i\leq k,\\ i \neq j} }\frac{\mu_i(p_j - p_i)}{|p_j - p_i|^3}\right)=0, \qquad \underline{j = 1, \ldots, k}.\label{eqn:cc-light}
\end{eqnarray}

Let us focus on (\ref{eqn:cc-light}). We fix a   $\qnm$. From Theorem~\ref{thm:limits} we expect that the light bodies will
form pleiades (clusters) close to points belonging to $\re{\qnm}$. Let us fix  $x^*\in \re{\qnm}$ and  consider a cluster  $\pkm = (p_1, \ldots, p_k)$ of $k$ light bodies around $x^*$, i.e. $p_j \approx x^*$ for $j=1,\dots,k$.
Since positions of heavy bodies are fixed, in the following part, we omit the argument $\qnm$ in $\vres$. We expand $\vres(p_i)$ into Taylor series at the point $x^*$ up to the second order term\footnote{We use notation
\begin{equation*}
(p_i - x^*)^T \frac{\partial^2}{\partial p^2}\vres(Q_N,x^*)(p_i - x^*)  = \frac{\partial^2}{\partial p^2}\vres(Q_N,x^*)(p_i - x^*)^2 =  \frac{\partial^2}{\partial p^2}\vres(Q_N,x^*)(p_i - x^*, p_i - x^*).
\end{equation*}
}
and by~(\ref{eq:rel-equi-restricted}) we obtain:
\begin{eqnarray}
  W(\qnm, \pkm) &=&  \sum_{j=1}^k \mu_j \left(\vres(x^*) + \frac{1}{2}\frac{\partial^2}{\partial p^2}\vres(x^*)(p_j - x^*)^2 + g(\qnm,x^*, p_j) \right) \nonumber\\
  & & + \sum_{1 \leq i < j \leq k} \frac{\mu_i \mu_j}{|p_i - p_j|},\label{eqn:pot-W}
\end{eqnarray}
where $g(\qnm, x^*, p_j)$ is a third order remainder term from the Taylor formula.
The function $g$ is trilinear in $x-p$, hence
\begin{equation*}
g(\qnm,x, p)=\bigo(|p - x|^3),
\end{equation*}
Moreover, we also have  $D_p g(\qnm, x, p) = \bigo(|p-x|^2)$ and $D_p^2 g(\qnm, x, p) = \bigo(|p-x|)$.

In order to study the limit of $\mu_i \to 0$, we scale the variables
\begin{subequations}
\begin{align}
  \Theta &= \sum_{j=1}^k \mu_j, \label{eq:SumM} \\
  \widetilde{\mu}_j & =  \frac{\mu_j}{\Theta},  \label{eq:m-tilde} \\
  \widetilde{p}_j & =  \frac{p_j - x^*}{\Theta^{1/3}}, \label{eq:q-norm}\\
  \widetilde{\C}_k&= (\widetilde{p}_1,\ldots, \widetilde{p}_k).
\end{align}
\label{eqn:scaling}
\end{subequations}
In the sequel we will use the inverse mapping of~(\ref{eq:q-norm}), which is a function of three parameters, thus 
\begin{equation}
p_j(\Theta,x, \widetilde{p}) = x + \Theta^{1/3}\widetilde{p_j}.  \label{eq:pj}
\end{equation}
In~(\ref{eqn:pot-W}) we use scaled values and we omit the constant term $\sum_{j=1}^k \mu_j \vres(x^*)$ (as it has no effect on the partial derivatives with respect to $p_j$ variables) and we obtain
\begin{eqnarray*}
 W(\qnm, \widetilde{P}_k^\mu) &=& \sum_{j=1}^k \Theta^{\frac{5}{3}} \frac{\widetilde{\mu}_j}{2} \frac{\partial^2}{\partial p^2}\vres(x^*)(\widetilde{p}_j, \widetilde{p}_j)
    + \sum_{j=1}^k \Theta^2\widetilde{\mu}_{j} g(\qnm,x^*, \widetilde{p}_j) \\
    & & + \sum_{1\leq i<j\leq k} \Theta^{\frac{5}{3}}\frac{\widetilde{\mu}_i \widetilde{\mu}_j}{|\widetilde{p}_i - \widetilde{p}_j|}\\
  & = & \Theta^{\frac{5}{3}}\left(  \widetilde{W}(x^*,\qnm, \widetilde{P}_k^\mu) + \sum_{j=1}^k \Theta^{\frac{1}{3}}\widetilde{\mu}_{j}  g(\qnm,x^*,\widetilde{p}_j)
\right),
\end{eqnarray*}

where
\begin{eqnarray}
\widetilde{W}(x^*,\qnm, \widetilde{P}_k^\mu) &=& \sum_{j=1}^k  \frac{\widetilde{\mu}_j}{2} \frac{\partial^2}{\partial p^2}\vres(\qnm,x^*)(\widetilde{p}_j,\widetilde{p}_j)  + \sum_{1\leq i<j\leq k}  \frac{\widetilde{\mu}_i \widetilde{\mu}_j}{|\widetilde{p}_i - \widetilde{p}_j|} \end{eqnarray}

We see that (\ref{eqn:cc-light}) becomes, for $j = 1, \ldots, k$

\begin{subequations}
\begin{align}
\frac{\partial \vres}{\partial p}(\qnm, x^*) & = 0\label{eqn:line-1} \\
 \frac{\partial }{\partial \widetilde{p}_j}\left( \widetilde{W}(x^*, \qnm, \widetilde{P}_k^\mu) +  \Theta^{\frac{1}{3}} \sum_{i=1}^k\widetilde{\mu}_{i}  g(\qnm, x^*, \widetilde{p}_i) \right) & =  0.
  \label{eq:cc-light-rescaled}
\end{align}
\label{eqn:all-lines}
\end{subequations}

\begin{definition}\label{def:cc-induced}
Let $\qnm = (q_1, \ldots, q_N)$ with masses $m = (m_i)_{i = 1, \ldots, N}$ and $x \in \reals^2$. Let $\mu = (\mu_j)_{j = 1, \ldots, k}$ be some constant positive real  numbers. The critical points of
\begin{eqnarray}
   \widetilde{W}(x,\qnm, \pkm) &=& \sum_{j=1}^k  \frac{\mu_j}{2} \frac{\partial^2\vres}{\partial p^2}(Q_N,x)(p_j,p_j)  + \sum_{1\leq i<j\leq k}  \frac{\mu_i \mu_j}{|p_i - p_j|} \label{eqn:new_coords}
\end{eqnarray}
 with fixed $x,\qnm$, i.e.\
 solutions $P_k = (p_j)_{j = 1, \ldots, k}$ of
 $$\frac{\partial \widetilde{W}}{\partial p_j}(x,\qnm, P_k^\mu)=0,\qquad \mbox{for } j=1,\dots,k$$
are called {\em CC of potential $\widetilde{W}$ induced by $(\qnm, x)$ for $\mu$}.
We write $P_k\in CC(\widetilde{W}, \qnm, x)$ for $\mu$.
\end{definition}

The following theorem is the main result in this section and the main motivation behind the notion of CC of potential $\widetilde{W}$ induced by $(\qnm, x)$. It can be seen as the generalization of the unstated explicitly  result from \cite{X}, where case $k=2$ was considered in Section 2.

\begin{theorem}
\label{thm:cont}
Assume that
\begin{itemize}
\item ${\hat Q}_N^m = (\hat q_1,\ldots, \hat q_N)$ is a non-degenerate normalized CC with masses $m = (m_1,\dots,m_N)$,
\item $x^*\in\re{{\hat Q}_N^m}$ and $x^*$ is non-degenerate,
\item ${\hat P}_k^{\widetilde{\mu}} = (\hat p_1,\ldots, \hat p_k)$  is  a non-degenerate central configuration of potential  $\widetilde{W}$  induced by $({\hat Q}_N^m, x^*)$ for some positive real numbers $\widetilde{\mu} = (\widetilde{\mu}_1,\ldots,  \widetilde{\mu}_k)$ and $\sum_i \widetilde{\mu}_i = 1$.

\end{itemize}
Then, for sufficiently small $\Theta>0$, there exists  a normalized non-degenerate central configuration: $$Q_{N+k} = (q_1,\dots, q_N, p_1, \ldots, p_k)$$
with masses $m_i$, for $i = 1, \ldots, N$ and $\mu_j = \Theta\widetilde{\mu}_j$, for $j = 1, \ldots, k$ such that
\begin{eqnarray*}
  q_i &=& \hat q_i + \bigo(\Theta),\\
  p_j &=& x^* + \Theta^{1/3}\hat p_j + \bigo(\Theta^{2/3}).
\end{eqnarray*}

\end{theorem}

\proof

\begin{figure}[H]
  \centering
  \begin{tikzpicture}

\draw[fill, veryLightGrey] (1.0, 2.0) circle (1.21cm);
\draw[dashed] (1.0, 2.0) circle (1.22cm);
\node[] at (-0.05, 0.9) {$\hat Q_N$};

\draw[fill, lightRed] (0.0, 2.0) circle (0.11cm);
\draw[fill] (0.0, 2.0) circle (0.05cm);
\node[] at (0.2, 1.9) {$\hat q_1$};

\draw[fill, lightRed] (1.0, 3.0) circle (0.11cm);
\draw[fill] (1.0, 3.0) circle (0.05cm);
\node[] at (0.8, 2.9) {$\hat q_2$};

\draw[fill, lightRed] (2.0, 2.0) circle (0.11cm);
\draw[fill] (2.0, 2.0) circle (0.05cm);
\node[] at (1.8, 1.9) {$\hat q_j$};

\draw[dotted] (1.0,1.0) -- (1.0, 1.0) arc(270:450:1.02);

\draw[fill, lightRed] (1.0, 1.0) circle (0.11cm);
\draw[fill] (1.0, 1.0) circle (0.05cm);
\node[] at (1.0, 1.25) {$\hat q_N$};

\draw[fill, red] (0.04, 2.03) circle (0.05cm);

\draw[fill, red] (1.05, 2.95) circle (0.05cm);

\draw[fill, red] (2.05, 1.95) circle (0.05cm);

\draw[fill, red] (0.96, 1.04) circle (0.05cm);

\draw[fill, lightRed] (3.0, 2.0) circle (0.25cm);
\node[] at (2.8, 1.6) {\textcolor{red}{$P_k$}};

\draw[fill, red] (2.82, 2.09) circle (0.01cm);
\draw[fill, red] (2.83, 1.9) circle (0.01cm);
\draw[fill, red] (2.95, 1.8) circle (0.01cm);
\draw[fill, red] (3.2, 1.92) circle (0.01cm);
\draw[fill, red] (3.08, 2.08) circle (0.01cm);
\draw[fill, red] (3.13, 1.81) circle (0.01cm);

\draw[fill] (3.0, 2.0) circle (0.05cm);
\node[] at (3.1, 2.25) {$x^*$};

\draw[->, dashed] (3.05, 2.0) .. controls (5.0, 2.0) and (3.0, 0.0) .. (5.2, 0.0);

\draw[fill, lightGrey] (5.8, -0.2) circle (0.6cm);
\node[] at (6.5, 0.4) {$\hat P_k$};

\draw[fill] (6.2, 0.0) circle (0.03cm);
\draw[fill] (5.7, 0.3) circle (0.03cm);
\draw[fill] (5.3, -0.2) circle (0.03cm);
\draw[fill] (5.5, -0.6) circle (0.03cm);
\draw[fill] (6.1, -0.5) circle (0.03cm);
\draw[fill] (5.85, -0.1) circle (0.03cm);
\draw[fill] (5.75, -0.4) circle (0.03cm);

\draw[dashed, red] (1.05, 2.05) circle (1.22cm);
\node[red] at (2.4, 2.9) {$Q_N$};

\end{tikzpicture}
\caption{We are given positions of ${\hat Q}_N^m\in {\cc{N}}$, $x^*\in\re{{\hat Q}_N^m}$ and  ${\hat P}_k^{\widetilde{\mu}}\in \cc{}(\widetilde{W}, {\hat Q}_N^m, x^*)$ for $\widetilde{\mu}$ marked in black. Light red circles symbolize balls around ${\hat Q}_N^m$ and  $x^*$. Dark red points indicate $\cc{N+k}$.}\label{fig:two-universes}
\end{figure}
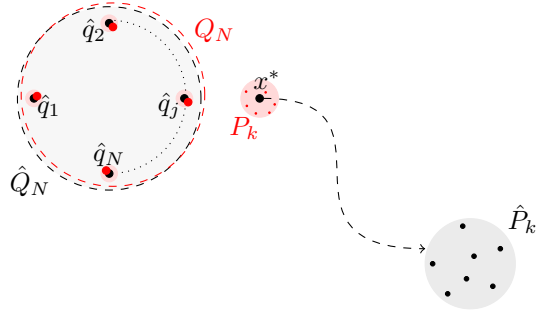

In the proof we will use the implicit function theorem applied to reduced system in the sense defined in Section~\ref{sec:red-sys}.
In view of Theorem~\ref{thm:non-deg-RS} we can assume that configuration ${\hat Q}_N^m$ is a solution of \RS with
 $\hat{q}_N$ being computed by the center of mass condition, $\hat{q}_{N-1,y}=0$, $\hat{q}_{N-1,x} \neq \hat{q}_{N,x}$.

In the proof we use the notation introduced in~(\ref{eqn:scaling}), therefore
from~(\ref{eq:cc-big-masses}) and~(\ref{eqn:all-lines}), it follows that a solution of the below system (\ref{eqn:ccNplusk}), with $\delta = \Theta^{1/3}$ as a parameter (to have analytic dependence on the parameter) and $q_i \in \mathbb{R}^2$, for $i=1,\dots,N-2$, $q_{N-1,x} \in \mathbb{R}$ , $x \in \mathbb{R}, \widetilde{p}_j \in \mathbb{R}^2$, $j=1,\dots,k$ as variables to be solved for,
gives rise to a solution of \RS for $(N+k)$-bodies with masses $m_i$ for $i=1,\dots,N$ and $\mu_j=\Theta \widetilde{\mu}_j$ for $j=1,\dots,k$
\begin{subequations}
\begin{align}
 \frac{\partial V}{\partial q_i}(\qnm) + \Theta\sum_{j=1}^k \widetilde{\mu}_j \frac{\partial \vres}{\partial q_i}(\qnm, p_j(\Theta,x,\widetilde{p}_j)) &=0, \qquad \underline{i=1,\dots, N-2}.  \label{eq:ift-big-masses}\\
  \frac{\partial V}{\partial q_{N-1,x}}(\qnm) + \Theta\sum_{j=1}^k \widetilde{\mu}_j \frac{\partial \vres}{\partial q_{N-1,x}}(\qnm, p_j(\Theta,x,\widetilde{p}_j)) &=0, \\
 \frac{\partial \vres}{\partial p}(\qnm,x) &= 0, \label{eqn:sys-big}\\
 \frac{\partial }{\partial \widetilde{p_j}}\left( \widetilde{W}(x,\qnm, \widetilde{P}_k) +  \Theta^{\frac{1}{3}} \sum_{i=1}^k\widetilde{\mu}_{i}  g(\qnm,x,\widetilde{p_i}) \right) &= 0, \qquad \underline{j=1,\dots, k},  \label{eq:ift-small-masses}
\end{align}
\label{eqn:ccNplusk}
\end{subequations}
where
\begin{equation}
  q_N(q_1,\dots,q_{N-1},p_1,\dots,p_j)=-\frac{1}{m_N}\left(\sum_{i=1}^{N-1} m_i q_i + \Theta \sum_{j=1}\tilde{\mu}_j p_j(\Theta,x,\tilde{p}_j)  \right)
\end{equation}

By the implicit function theorem\footnote{ Let us denote by $F$ the right hand side of system (\ref{eqn:ccNplusk}). Let $\Omega$ be is an open set in $\reals^{2N-1+2k} \times \reals$.  Let $F\colon \Omega\to \reals^{2N-1+2k}$ be the system~(\ref{eqn:ccNplusk}). We know  that $F\in C^1$ and for $(\hat Q_N, x^*, \hat P_k, \delta=0)\in \Omega$ it holds $F(\hat Q_N, x^*, \hat P_k, \delta=0) = 0$. Moreover
$\frac{\partial{F}}{\partial (Q_N, x, \widetilde{P}_k)}(\hat Q_N, x^*, \hat P_K)$ is an isomorphism. Note that since $(\partial{F})/(\partial (Q_N, x, \widetilde{P}_k))(\hat Q_N, x^*, \hat P_K)$ is lower triangular matrix, it is enough to have non-degenerate blocks on the diagonal, i.e.\ $(\partial^2 V)/(\partial q_i^2)(\hat Q_N, x^*, \hat P_K)$, $(\partial^2 \vres)/(\partial p^2)(\hat Q_N, x^*, \hat P_K)$, $(\partial^2 \widetilde{W})/(\partial \widetilde{p}_j^2)(\hat Q_N, x^*, \hat P_K)$.

Then there exists $\gamma >0$, $\eta > 0$ and  a smooth
function
$\varphi$ (note that in our case the function $\varphi$ is smooth for $\delta = \Theta^{1/3}$)
such that: $\mbox{Ball}(0, \gamma)\times \mbox{\rm Ball}((\hat Q_N, x^*, \hat P_K), \eta)\subset \Omega$ and for all
$|\delta| < \gamma$, the only point $(Q_N, x, \widetilde{P}_k) \in {\rm Ball}((\hat Q_N, x^*, \hat P_k), \delta)$ satisfying $F(Q_N, x, \widetilde{P}_k, \delta) = 0$ is
$(Q_N, x, \widetilde{P}_k) = (Q_N(\delta), x(\delta), \widetilde{P}_k(\delta)) = \varphi(\delta)$.},
we know that  for sufficiently small $\delta > 0$, the system~(\ref{eqn:ccNplusk}) has unique, nondegenerate solution. These solutions are
\begin{eqnarray*}
  q_i(\delta) & = & \hat q_i+ \bigo(\delta^3), \\
  x(\delta) &=& x^* + \bigo(\delta^3), \\
    \widetilde{p}_j(\delta) &=& \hat p_j + \bigo(\delta).
\end{eqnarray*}

Written in terms of $\Theta$ we have
\begin{eqnarray*}
  q_i(\Theta) & = & \hat q_i+ \bigo(\Theta), \\
  x(\Theta) &=& x^* + \bigo(\Theta), \\
    \widetilde{p}_j(\Theta) &=& \hat p_j + \bigo(\Theta^{1/3}).
\end{eqnarray*}

From the above, (\ref{eq:pj})  we obtain
\begin{eqnarray*}
  p_j(\Theta) & =& x(\Theta) + \Theta^{1/3}\widetilde{p}_j(\Theta)\\
  & = & \left(x^* + \bigo(\Theta) \right) + \Theta^{1/3}\left(\hat p_j + \bigo(\Theta^{1/3}) \right)\\
  & = & x^* + \Theta^{1/3}\hat p_j + \bigo(\Theta^{2/3}).
\end{eqnarray*}

Observe that
\begin{eqnarray*}
  x_{N-1}(\Theta) - x_N(\Theta)&=& \hat{x}_{N-1} + \bigo(\Theta) - \left(\frac{1}{m_N}\left(\sum_{i=1}^{N-1} m_i x_i + \Theta \sum_{j=1}\tilde{\mu}_j p_{j,x}(\Theta)  \right)  \right) \\
    &=&\hat{x}_{N-1} - \hat{x}_N + \bigo(\Theta).
\end{eqnarray*}
Therefore for sufficiently small $\Theta$ from Theorem~\ref{thm:red-to-full} it follows that our solution of \RS is a normalized central configuration.

The configuration is non-degenerate by Theorem~\ref{thm:RS-non-deg-CC},  because it is a non-degenerate solution of \RS and $x_{n-1}\neq x_n$.

\qed


\section{The limit equation for light bodies, some a priori bounds}
\label{sec:restr-cc-light-bodies}

In this section we start  investigation of  the critical  points of the potential $\widetilde{W}(x^*,\qnm, \pkm)$ with    fixed $x^*$ and $\qnm$. Since we focus on the light bodies, we drop $(x^*,\qnm)$ and all tildes in notation.
By a suitable rotation we diagonalize a quadratic form $ \frac{\partial^2}{\partial p^2}\vres(Q_N,x^*)$, which becomes
\begin{equation}
  \frac{\partial^2}{\partial q^2}\vres(Q_N,x^*)((x,y),(x,y))=\frac{a x^2}{2} + \frac{b y^2}{2}
\end{equation}
and we use the following notation
 \begin{equation}
 p_i=(x_i,y_i), \quad r_{ij}=|p_i - p_j|.
 \end{equation}

\begin{definition}
\label{def:ccrbp}
{\em Central configurations of restricted $k$-light body problem}
\ccrbp{k}   are the critical points of the following potential
\begin{eqnarray}
 W(\pkm)&=& \sum_i \frac{\mu_i}{2}(a x_i^2 + b y_i^2)   + \sum_{i<j} \frac{\mu_i \mu_j}{|p_i - p_j|}.
    \label{eq:V-expansion}
\end{eqnarray}
Alternatively, we will use the name \emph{the central configuration problem in anisotropic space for $k$-bodies}.
\end{definition}

We obtain for \ccrbp{k} the following system of equations
\begin{equation}\label{eqn:a-b}
\boxed{
\begin{aligned}
  \mu_i a x_i &=&  \sum_{j, j \neq i} \frac{\mu_i \mu_j (x_i - x_j)}{r_{ij}^3}, \qquad \forall i\in \{1, \ldots, k\} \\
     \mu_i b y_i &=&  \sum_{j, j \neq i} \frac{\mu_i \mu_j (y_i - y_j)}{r_{ij}^3}, \qquad \forall i\in \{1, \ldots, k\}.
\end{aligned}}
\end{equation}

Observe also that for $a=b>0$ we obtain the problem normalized central configurations (with $a$ being an inverse of the gravitational constant).
If $a \neq b$, then the rotational symmetry is broken, so system (\ref{eqn:a-b}) can be seen as \emph{the central configuration problem in anisotropic space}.

System of equations~(\ref{eqn:a-b}) has the same scaling properties as the system of equations (\ref{eq:cc}), hence we can scale masses and variables $(x_i,y_i)$
so that  $\sum_i \mu_i=1$ (this is our choice in CAP discussed later in the paper).  We can also scale $a,b$ and variables $(x_i,y_i)$ to  obtain $|a|=1$ (or $|b|=1$), hence  only the ratio $\frac{a}{b}$ matters.

After adding equations~(\ref{eqn:a-b}) for $x_i$'s and $y_i$'s  separately we obtain
\begin{equation}
  c=\sum_i \mu_i p_i =0.   \label{eq:cc-cofmass}
\end{equation}
We see that the center of mass for solutions for (\ref{eqn:a-b}) is at the origin, just as in the case of \cc{N}.

\subsection{Moment of inertia-like quantities}
\label{subsec:mofi}
This is an adaptation of the well known identities for central configurations (see \cite[Sec. 2 and 3]{MZ}) to the present context.

\begin{lemma}
\label{lem:ccab-mofin}
Assume that
\begin{equation}
  M=\sum_i \mu_i=1. \label{eq:sum-mui}
\end{equation}
For solutions of~(\ref{eqn:a-b})
hold
\begin{subequations}
\begin{align}
  a \left( \sum_i \mu_i x_i^2 \right) = a\left(\sum_{i<j} \mu_i\mu_j(x_i - x_j)^2 \right) & =  \sum_{i<j} \frac{\mu_i\mu_j}{r_{ij}^3}(x_i - x_j)^2\label{eqn:ax}\\
b \left( \sum_i \mu_i y_i^2 \right) =  b\left(\sum_{i<j} \mu_i\mu_j(y_i - y_j)^2 \right) & =  \sum_{i<j} \frac{\mu_i\mu_j}{r_{ij}^3}(y_i - y_j)^2\label{eqn:by}\\
  a \left( \sum_i \mu_i x_i^2 \right) +  b \left( \sum_i \mu_i y_i^2 \right)&=\sum_{i<j} \frac{\mu_i \mu_j}{r_{ij}}.\label{eq:aplusb}
\end{align}
\end{subequations}
\end{lemma}
\begin{proof}
See Appendix~\ref{inertia-proof}.
\end{proof}

Since the rightmost expressions in (\ref{eqn:ax}, \ref{eqn:by}) are positive we immediately obtain the following conclusion.
\begin{conclusion}
\label{thm:nosol}
The number of solutions of~(\ref{eqn:a-b}) is:
\begin{itemize}
\item If $a,b \leq 0$, then there are no solutions.

\item If $a\leq 0$ and $b>0$, then all solutions
are on \OY\ ($\forall i\colon x_i=0$). By the Moulton's Theorem (see~\cite{Mou,Sm}), there are $k!$ solutions, one for each ordering of the bodies  on \OY.

\item If $a>0$ and $b\leq 0$, then all solutions
are on \OX\ ($\forall i\colon y_i=0$).  By the Moulton's Theorem (see~\cite{Mou,Sm}) there are $k!$ solutions, one for each ordering of the bodies
  on \OX.
\end{itemize}
\end{conclusion}

Thus the only non-trivial case is $a, b > 0$. In this case,  from the Moulton Theorem, we also have collinear solutions on $x$ and $y$ axes.

\subsection{Lower bound for the distance between bodies}
Just like in \cite{MZ} from the  Lemma~\ref{lem:ccab-mofin} we derive  the lower bound for $r_{ij}$ provided that we have an upper bound for $I=\sum_i \mu_i r_i^2$.

\begin{lemma}\label{lm:lower-bnd}
Assume (\ref{eq:sum-mui}) and $a,b>0$.
If $\pkm$ is a solution of~(\ref{eqn:a-b}),
then for any bodies $p_i, p_j$ in $\pkm$
\begin{equation}
r_{ij} \geq \frac{\mu_i\mu_j}{\max\{a, b\}R^2 },
\end{equation}
where $R=\max_i |p_i|$.
\end{lemma}

\begin{proof}
See Appendix~\ref{lower-bound-proof}
\end{proof}

\subsection{Upper bound for the size}

This is an adaptation of the upper bound on the size  of the normalized central configurations from \cite[Sec. 3.2]{MZ} recalled in the present paper as Theorem~\ref{thm:ncc-upperBound}.

From Conclusion~\ref{thm:nosol} it follows we cannot have both $a\leq 0$ and $b \leq 0$,
hence we assume $a>0$ or $b >0$.

\begin{theorem}
\label{thm:upp-bnd}
Assume that $a>0$ or $b >0$.
Assume $\pkm$ is a solution of~(\ref{eqn:a-b}) and $M=\sum_i \mu_i$.
If $a>0$, then
\begin{equation}
  \max_{i} |x_i| \leqslant  \left\{
                                   \begin{array}{ll}
                                     (n-1) \left(\frac{M}{a}\right)^{1/3}, & \hbox{$n\geqslant  2$;} \\
                                     \left(2^{1/3}+2^{-2/3} \right) (n-2)^{2/3}\left( \frac{M}{a}\right)^{1/3}, & \hbox{$n\geqslant  4$.}
                                   \end{array}
                                 \right.
\end{equation}
and if $b>0$, then
\begin{equation}
  \max_{i} |y_i|  \leqslant  \left\{
                                   \begin{array}{ll}
                                     (n-1) \left(\frac{M}{b}\right)^{1/3}, & \hbox{$n\geqslant  2$;} \\
                                     \left(2^{1/3}+2^{-2/3} \right) (n-2)^{2/3}\left( \frac{M}{b}\right)^{1/3}, & \hbox{$n\geqslant  4$.}
                                   \end{array}
                                 \right.
\end{equation}
\end{theorem}
\begin{proof}
See Appendix~\ref{upp-thm-proof}
\end{proof}


\section{Analytical solutions  for CCs in anisotropic plane}
\label{sec:analytic-sol}

\subsection{Collinear solutions}

As it was noticed at the end of Subsection~\ref{subsec:mofi}, by the Moulton Theorem \cite{Mou,Sm}, there are $k!$ collinear solutions on both $x$ and $y$ axes.
The next theorem states  that if $a \neq b$ there are no other collinear solutions.
\begin{theorem}
\label{thm:coll-on-coord-axis}
Assume that $(p_1, \ldots, p_k)$ is collinear CC for the restricted problem~(\ref{eqn:a-b})
with $a \neq b$. Then either $x_i=0$ for all  $i=1,\ldots,n$, or  $y_i=0$ for all $i=1,\dots, n$.
\end{theorem}

\begin{proof}
See Appendix~\ref{coll-proof}.
\end{proof}

\subsection{$k=2$}
For $k=2$ all the solutions are collinear and contained in coordinate axes.

\subsection{$k=3$ with equal masses}

For $k=3$  by the Moulton Theorem we have $2\cdot 3!$  collinear solutions contained in coordinate axes.

It appears   that there exists only one type of non-collinear solutions described in the next subsection. The computer assisted proof reported in Section~\ref{sec:cc2k} confirms this for a particular values $a=3/4$ and $b=9/4$.

\subsubsection{The isosceles triangle}

\begin{theorem}
\label{thm:iso-triangle}
Assume that $a,b>0$ and $\mu_1=\mu_2=\mu_3=\mu$ and
\begin{equation}
\frac{b}{a} > \displaystyle{\frac{5}{12}} .\label{eqn:assumption}
\end{equation}
 Then there exists a solution of~(\ref{eqn:a-b})
which forms an isosceles triangle symmetric with respect to \OX. Moreover, if we denote by $s$ the side and by $r$ the base of the triangle, then
\begin{itemize}
\item $s = \displaystyle{\left(\frac{3\mu}{a}\right)^{1/3}}$
\item $r = \displaystyle{\left(\frac{2\mu}{b-\frac{a}{3}}\right)^{1/3}}$.
\end{itemize}

\end{theorem}

\begin{proof}
See Appendix~\ref{triangle-proof}.
\end{proof}

The geometry of the triangle is given by
\begin{eqnarray*}
  \frac{s}{r}=\left(\frac{3}{2}   \frac{b-\frac{a}{3}}{a} \right)^{1/3}=\left(\frac{3}{2}\left(\frac{b}{a}-\frac{1}{3}\right) \right)^{1/3}.
\end{eqnarray*}

From Theorem~\ref{thm:iso-triangle} we obtain two isosceles triangles: one is obtained from the other by the reflection with respect to \OY.
If   $\frac{12}{5} > \frac{b}{a} > \frac{5}{12}$ then from Theorem~\ref{thm:iso-triangle} follows that we have two isosceles triangles with the reflectional symmetry with respect to \OX\ and
two other  isosceles triangles with the reflectional symmetry with respect to \OY\ (see Figure~\ref{fig:triangles}).
For $\frac{b}{a}=\frac{5}{12}$ the isosceles triangle becomes singular and becomes the collinear solution on the \OY.

\begin{minipage}[h]{.95\textwidth}
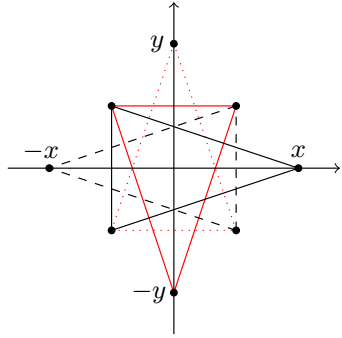
\begin{figure}[H]
  \centering
\resizebox{0.3\textwidth}{!}{%
\begin{tikzpicture}
\draw[thin, red] (0.75, 0.75) -- (0.0, -1.5);  
\draw[thin, red] (-0.75, 0.75) -- (0.0, -1.5);  
\draw[thin, red] (-0.75, 0.75) -- (0.75, 0.75);  

\draw[fill] (0.0, -1.5) circle (0.4mm);   // p2

\node[] at (-0.3, -1.5) {\small $-y$};

\draw[dotted, red] (0.75, -0.75) -- (0.0, 1.5);  
\draw[dotted, red] (-0.75, -0.75) -- (0.0, 1.5);  
\draw[dotted, red] (-0.75, -0.75) -- (0.75, -0.75);  

\draw[fill] (0.0, 1.5) circle (0.4mm);   // p2

\node[] at (-0.2, 1.5) {\small $y$};

\draw[fill] (1.5, 0.0) circle (0.4mm);   // p2
\draw[fill] (-0.75, 0.75) circle (0.4mm); // p3
\draw[fill] (-0.75, -0.75) circle (0.4mm);  // p4

\node[] at (1.5, 0.2) {\small $x$};

\draw[] (-0.75, 0.75) -- (1.5, 0.0);  
\draw[] (-0.75, -0.75) -- (1.5, 0.0);  
\draw[] (-0.75, -0.75) -- (-0.75, 0.75);  

\draw[fill] (-1.5, 0.0) circle (0.4mm);   // p2
\draw[fill] (0.75, 0.75) circle (0.4mm); // p3
\draw[fill] (0.75, -0.75) circle (0.4mm);  // p4

\node[] at (-1.6, 0.2) {\small $-x$};

\draw[dashed] (0.75, 0.75) -- (-1.5, 0.0);  
\draw[dashed] (0.75, -0.75) -- (-1.5, 0.0);  
\draw[dashed] (0.75, -0.75) -- (0.75, 0.75);  

\draw[->, thin] (-2.0, 0.0) -- (2.0, 0.0);
\draw[->, thin] (0.0, -2.0) -- (0.0, 2.0);

\end{tikzpicture}
}
\caption{Orientations of isosceles triangles: two black for $b/a > 5/12$ and two red for $a/b > 5/12$.}\label{fig:triangles}
\end{figure}
\end{minipage}

\subsection{$k=4$ with equal masses}

For $k=4$  by the Moulton Theorem we have $2 \cdot 4!$ collinear solutions contained in  coordinate axes. From  Theorem~\ref{thm:coll-on-coord-axis} we see that there are no other collinear solutions.

 We can analytically prove the existence of two types of non-collinear solutions for equal mass case: a rhombus and a rectangle, both symmetrical with respect to $x$ and $y$ axes. There exist other types of solutions, see  Section~\ref{subsubsec:full-list} for the computer assisted proof of full list
 for equal mass case and $a=3/4$ and $b=9/4$.

\subsubsection{Rhombus}

\begin{theorem}\label{thm:rhombus}
Assume $a,b >0$.
Assume that $\sum_{i= 1}^4 \mu_i = 1$ and $\mu_i = \mu$ for $i = 1, 2, 3, 4$.
Then there exists a solution of~(\ref{eqn:a-b})
which forms a rhombus (consisting from points $(\pm x,0),(0,\pm y)$, for some positive $x,y$) symmetrical with respect to $x$ and $y$ axes.
Moreover, if we denote by $k=\frac{y}{x}$ the ratio of diagonals  and by $r$ the length of the sides, then
\begin{itemize}
\item $k$  is a unique solution to
\begin{equation}\label{eqn:k}
\frac{8k^3 + (k\sqrt{1 + k^2})^3}{8k^3 + (\sqrt{1 + k^2})^3} = \frac{a}{b}
\end{equation}
\item  $r > \displaystyle{\left(\frac{9\mu}{4\min\{a, b\}}\right)^{1/3}}$
\item  in the above range, $r$ is a unique solution to
\begin{equation}\label{eqn:m}
\left(\frac{2\mu}{ar^3 - 2\mu} \right)^{\frac{2}{3}} +  \left(\frac{2\mu}{br^3 - 2\mu} \right)^{\frac{2}{3}} = 4.
\end{equation}
\end{itemize}
\end{theorem}

\begin{proof}
See Appendix~\ref{rhombus-proof}.

\end{proof}

\subsubsection{Rectangle}

\begin{theorem}\label{thm:rectangle}
Assume $a,b >0$.
Assume that $\sum_{i= 1}^4 \mu_i = 1$ and $\mu_i = \mu$ for $i = 1, 2, 3, 4$.
Then there exists  solution of~(\ref{eqn:a-b})
which forms a rectangle (consisting of points of the form $(\pm x, \pm y)$ for some positive $x,y$)   symmetrical with respect to $x$ and $y$ axes. Moreover, if we denote by $d = 2r$ the diagonal and by $\varphi$ the angle between diagonal and the $x$ axis, then
\begin{itemize}
\item $\varphi$  is a unique solution to
\begin{equation}\label{eqn:phi}
\tan^3(\varphi) \frac{\cos^3(\varphi) + 1}{\sin^3(\varphi) + 1} = \frac{a}{b}
\end{equation}
\item  $r > \displaystyle{\left(\frac{\mu}{2\min\{a, b\}}\right)^{1/3}}$
\item  in the above range, $r$ is a unique solution to
\begin{equation}\label{eqn:mu}
\left(\frac{\mu}{4ar^3 - \mu} \right)^{\frac{2}{3}} + \left(\frac{\mu}{4br^3 - \mu} \right)^{\frac{2}{3}} = 1.
\end{equation}
\end{itemize}
\end{theorem}

\begin{proof}
See Appendix~\ref{rectangle-proof}.
\end{proof}


\section{Computer assisted results for CCs with heavy and light bodies and corresponding  problems in anisotropic plane}
\label{sec:cap}

In this section we discuss and  compare results obtained from two programs, first is the program for central configurations (PGU)  described in~\cite{MZ} and the second program (PHU), which  solves the problem  defined in Section~\ref{sec:restr-cc-light-bodies}.

\subsection{PGU ---  about the program }
The program for rigorous computations of all CC for N-body problem for fixed set of masses $m_i$ is described in \cite{MZ}.  The program, if successful,  returns all central configurations. It solves rigorously the reduced system  \RS (discussed in Section~\ref{sec:red-sys}) and the necessary condition for the program to succeed is that all solutions of \RS are non-degenerate. 

On input, the program accepts a box in the configuration space of the reduced system \RS and then uses binary subdivision to produce boxes (covering the initial box)
which satisfy one of the following conditions:
\begin{itemize}
\item[(C1)] there is no solution in a box
\item[(C2)] there is exactly one solution (non-degenerate) in a box
\item[(C3)] cannot decide rigorously if (C1)  or (C2) is satisfied and the size of the box is below some threshold
\end{itemize}
The program always stops. If there are no undecided boxes, i.e.\  satisfying (C3), then we know that all normalized central configuration are in boxes satisfying (C2).

\subsection{PHU ---  about the program}

PHU  looks for all solutions of system of equations~(\ref{eqn:a-b}) for given $a,b$ and masses $\mu_i>0$ such that $\sum \mu_i=1$.
 This allows us to obtain \ccrbp{k} with the coefficients $a,b$ (see Section~\ref{sec:restr-cc-light-bodies}).

The basic idea of the program is the same as in the case of PGU: for a given initial box  we use the binary subdivision algorithm to produce a set of boxes satisfying
conditions (C1)--(C3) and covering the initial box. The initial box was chosen to contain all possible solutions according to the upper bound  on the size of solution from Theorem~\ref{thm:upp-bnd}.  If it turns out that there are no boxes satisfying (C3), then we obtained all solutions
of system of equations~(\ref{eqn:a-b}) and all solutions are non-degenerate.

In the program we use the following tests:
\begin{itemize}
\item using the Krawczyk operator (discussed in \cite{MZ}) we can: decide that a box satisfies (C1) or (C2) or we can reduce the size of the box
\item for boxes containing pairs of close bodies we check  the condition from Lemma~\ref{lm:lower-bnd}, where we obtained  lower bound for the distance between bodies. This test  allows us to remove boxes with very close bodies during the binary subdivision process.
\end{itemize}


\subsection{Two heavy and $k$-light bodies}
\label{subsec:CC-2+k}

We perform our computations for $(2+k)$-problem with $k=3,4$.  We are interested in the neighborhood of relative equilibria
in the restricted $3$-body problem. It is well known (see for example \cite{MD}) that in the restricted $3$-body problem we have five relative equilibria: three collinear  with the primaries (called $L_1,L_2,L_3$) and two triangular ones denoted by $L_4$ and $L_5$, which form an equilateral triangle with the primaries. $L_4$ and $L_5$ are mutually symmetric with respect to the line passing through the primaries.  $L_4$ and $L_5$ are mutually symmetric with respect to the line passing through the primaries. See Figure~\ref{fig:lagrange-points}.
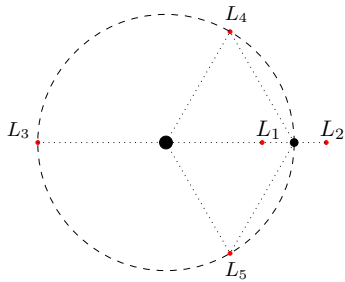
\begin{figure}[H]
  \centering
\resizebox{0.3\textwidth}{!}{%
\begin{tikzpicture}
\draw[fill] (1.5, 1.5) circle (1mm);   // h1
\draw[fill] (3.5, 1.5) circle (0.6mm);   // h2

\draw[fill, red] (3.0, 1.5) circle (0.3mm);   // L1
\node[] at (3.1, 1.7) {$L_1$};

\draw[fill, red] (4.0, 1.5) circle (0.3mm);   // L2
\node[] at (4.1, 1.7) {$L_2$};

\draw[fill, red] (-0.5, 1.5) circle (0.3mm);   // L3
\node[] at (-0.8, 1.7) {\small $L_3$};

\draw[fill, red] (2.5, 3.23) circle (0.3mm);   // L4
\node[] at (2.6, 3.5) {\small $L_4$};
\draw[dotted] (1.5, 1.5) -- (2.5, 3.23);
\draw[dotted] (3.5, 1.5) -- (2.5, 3.23);

\draw[fill, red] (2.5, -0.23) circle (0.3mm);   // L5
\node[] at (2.6, -0.5) {\small $L_5$};
\draw[dotted] (1.5, 1.5) -- (2.5, -0.23);
\draw[dotted] (3.5, 1.5) -- (2.5, -0.23);

\draw[dashed] (1.5, 1.5) circle (2.0);   
\draw[dotted] (-0.5, 1.5) -- (4.0, 1.5);
\end{tikzpicture}
}
\caption{Two heavy bodies (primaries) are depicted as black dots. Light bodies (red) at the Lagrange points ($L_1, \ldots, L_5$) are in equilibrium.}\label{fig:lagrange-points}
\end{figure}  
It turns out that for collinear relative equilibria $L_i$, $i=1,2,3$ we have $a>0$ and $b<0$, which in view of Conclusion~\ref{thm:nosol} makes
the \ccrbp{k}  problem not interesting.
We focus on $L_4$ (for $L_5$ the situation is symmetric). In Lemma~\ref{lem:L4eigenVal-eq} we computed values of $a$ and $b$ for the equal masses
case. These are
\begin{equation}
   a=\frac{3}{4}, \qquad b=\frac{9}{4}.
\end{equation}
We will use them in the computations PHU computations reported below. For PHU computations we assume that we have
\begin{equation}
\sum_{i=1}^k \mu_i=1, \quad \mu_i=\frac{1}{k}.
\end{equation}

In this paper we run PGU with two heavy masses close to $1/2$  and $k$-light bodies with small masses, for $k=3,4$.
The heavy masses are located close to $(\pm 0.5,0)$ and the light bodies were located in box around $L_4$ 
(see  Appendix~\ref{sec:L4ab}).

We refer to the notation introduced in Section~\ref{sec:red-sys} and we obtain \RS in the following way. First, we set $n = N+k$.  Then we order bodies: $(q_1, \ldots, q_k)$ are light bodies, $q_{k+1} = q_{n-1}$ and $q_{k+2} = q_n$ are heavy bodies. We set $k_1= n-1$, this means that one heavy body is placed on the $x$ axis and the second one is computed from the center of mass condition~(\ref{eq:cc-kart-n-th}).

\subsubsection{A comparison of PGU and PHU}

By Theorem~\ref{thm:cont} we know that the non-degenerate relative equilibria of a restricted $(N+1)$-body problem  combined with non-degenerate solutions for $k$ bodies in the $W$-potential can be extended to solutions of $Q_{N+k}$ central configuration problem.  In the future work we hope
do rigorously this continuation, so that we will have a rigorous result for central configuration for an explicit range of small masses from zero
up to some macroscopic quantity, from which cc could be continued using other tools, like Krawczyk operator.

However, for a present moment, we will just compare the results from both programs. For this end we identify solutions from PHU  with those from PGU as follows.

Following the coordinate change from   Theorem~\ref{thm:cont}  (see (\ref{eq:m-tilde}, \ref{eq:q-norm}))  for solutions of PGU we compute
\begin{subequations}
\begin{align}
\tilde{p_i} & =  \frac{p_i - c}{\Theta^{1/3}}\label{eqn:coord-norm}\\
\tilde{\mu_i} & =  \frac{\mu_i}{\Theta},\label{eqn:mass-norm}
\end{align}\label{eqn:normalization}
\end{subequations}
where $\tilde{p_i}$ are normalized coordinates $p_i$ of small bodies obtained by PGU, $\mu_i$ --- masses of small bodies before normalization, $\Theta = \sum \mu_i$ and $c$ is a center of masses $\mu_i$.

Observe that in (\ref{eq:q-norm}) we have $x^*$ while in (\ref{eqn:coord-norm}) we have $c$ - the center of masses of light bodies.
However from Theorem~\ref{thm:cont}  it follows that
\begin{equation}
  c=x^* + \bigo(\Theta^{2/3}).
\end{equation}

The obtained values $\tilde{p}_i$ are then compared to those from PHU.  The agreement is quite satisfactory for the values of masses considered,
thus supporting our hope for the possibility of rigorous continuation.

In our discussion below we give approximate values, in fact we use midpoints of exact bounds obtained by our programs. The width of those bounds are very small, therefore it makes sense to using midpoints in comparison. For exact bounds refer to report files attached in the Appendix~\ref{app:reports}. 

 Below, we also present demonstrative pictures of solutions. Pictures are created in {\em Mathematica}~\cite{Mth} with midpoints of the intervals bounding the soultion. Labels of bodies in $ab$-potential are colored blue, in $(2+k)$-body problem -- red. The role of ordering and colors of bodies is only to illustrate the compatibility of solutions.

\subsection{Central configurations in $(2+k)$-body problem}
\label{sec:cc2k}

\subsubsection{$k=3$}
PGU is run for data with two big masses $m_1=m_2 = \frac{1}{2}\left(1- \frac{3}{4} \cdot 10^{-7}\right)$ and small masses  $\mu=1/4 \cdot 10^{-7}$, while PHU for masses $\tilde{\mu}=1/3$.  The initial boxes for light bodies in PGU are placed at $[-0.2, 0.2]\times [0.7, 1.0]$, first heavy body is placed at $[0.49, 0.65]\times 0$, second is computed by the center of mass.
In both programs PHU and PGU we obtained the same number of solutions and we were able to pair to the solutions from both programs as suggested
by Theorem~\ref{thm:cont}.

We obtained $4 \cdot 3!=24$ solutions,
\begin{itemize}
\item collinear on \OX, $3!$ solutions
\item collinear on \OY, $3!$ solutions
\item isosceles triangle symmetric with respect to \OX\ and its reflection with respect to \OY, $2 \cdot 3!$ solutions
\end{itemize}

\subsubsection*{Comparison of solutions}
\label{subsubsec:full-list-three}
\begin{itemize}
\item Figure~\ref{fig:coll3} displays horizontal and vertical collinear solutions.

\addtocounter{figure}{-1}  
\begin{figure}[H]
\centering
\begin{subfigure}{0.42\textwidth}
  \centering
\includegraphics[scale=0.7]{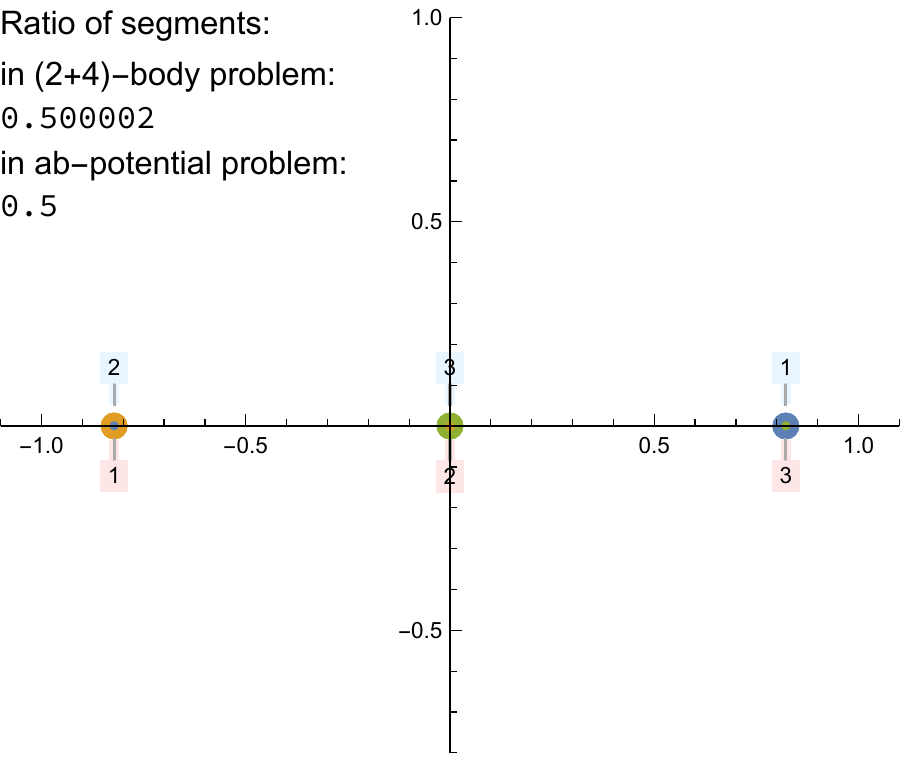}
\caption{Horizontal.}
\end{subfigure}%
\qquad
\begin{subfigure}{0.42\textwidth}
  \centering
\includegraphics[scale=0.7]{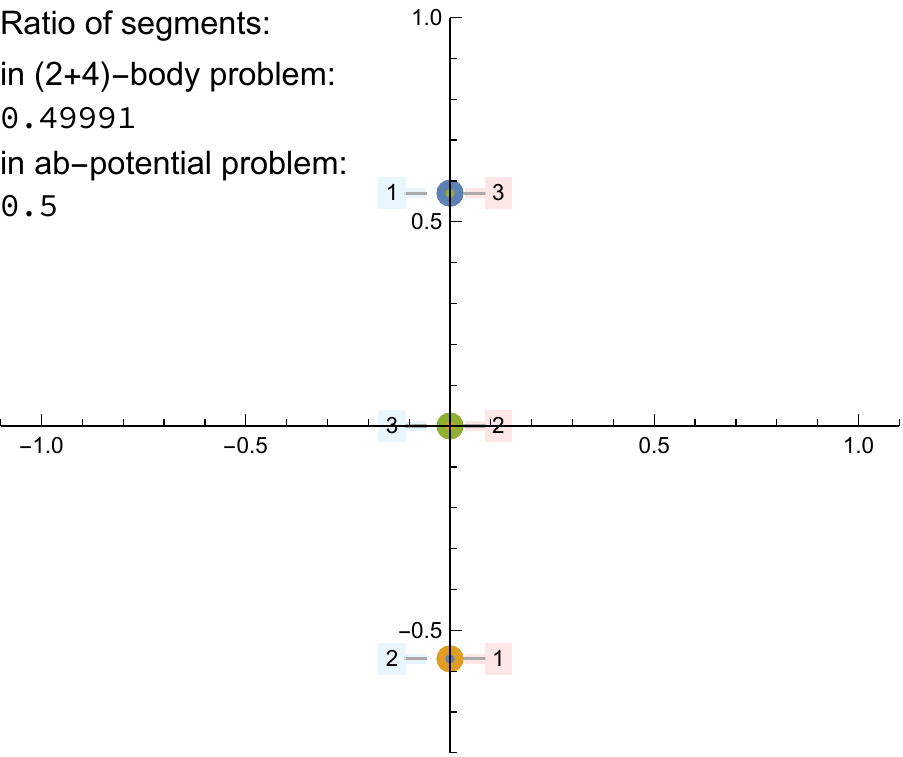}
 \caption{Vertical.}
\end{subfigure}
 \captionof{figure}{Collinear solutions.}\label{fig:coll3}
\end{figure}

Coordinates of the collinear solutions in the (2+3)-body problem, normalized by~(\ref{eqn:normalization})  and coordinates in $ab$-potential are given below:
{
\renewcommand{\arraystretch}{1.5}
$$
\begin{array}{m{1em}|c|rr|rr}
& q_i &\multicolumn{2}{c}{\mbox{PGU normalized coordinates}} & \multicolumn{2}{c}{\mbox{PHU coordinates (in $ab$-potential)}}\\
\hline
\multirow{2}{*}{\rotatebox{90}{horizontal~~}} & q_1 & (-0.822075, & 0.000685358) & (-0.8220706914, & 0.0000000000)\\
& q_2 & (0.822075, & 0.000685548) &  (0.8220706914, & 0.0000000000)\\
& q_3 & (-4.17619\times10^{-17}, & -0.00137091) & (0.0000000000, & -0.0000000000)\\
\hline
\multirow{2}{*}{\rotatebox{90}{vertical~~~}} & q_1 & (7.11379\times10^{-8}, & -0.569958) & (0.0000000000, & 0.5699919822)\\
& q_2 & (-7.11379\times10^{-8}, & 0.570128) & (-0.0000000000, & -0.5699919822)\\
& q_3 & (0.0, & -0.0000681896) & (0.0000000000, & 0.0000000000)\\
\end{array}
$$
}

\item Figure~\ref{fig:triangle3} displays the Isosceles triangle solutions found.

\addtocounter{figure}{-1}  
\begin{figure}[H]
\centering
\begin{subfigure}{0.42\textwidth}
  \centering
\includegraphics[scale=0.7]{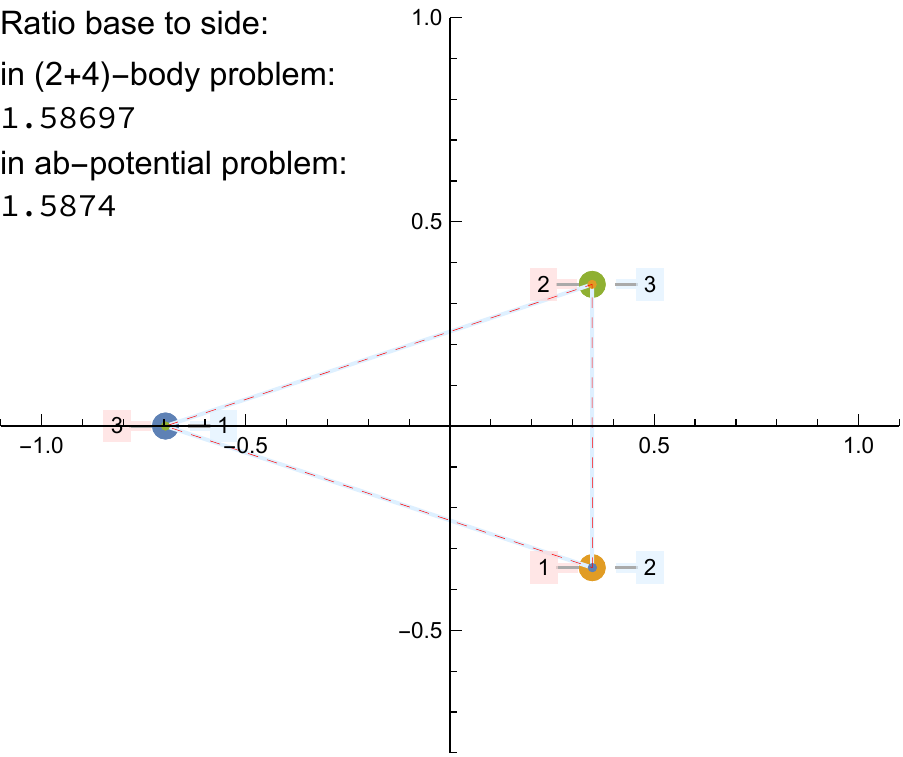}
\end{subfigure}%
\qquad
\qquad
\begin{subfigure}{0.42\textwidth}
  \centering
\includegraphics[scale=0.7]{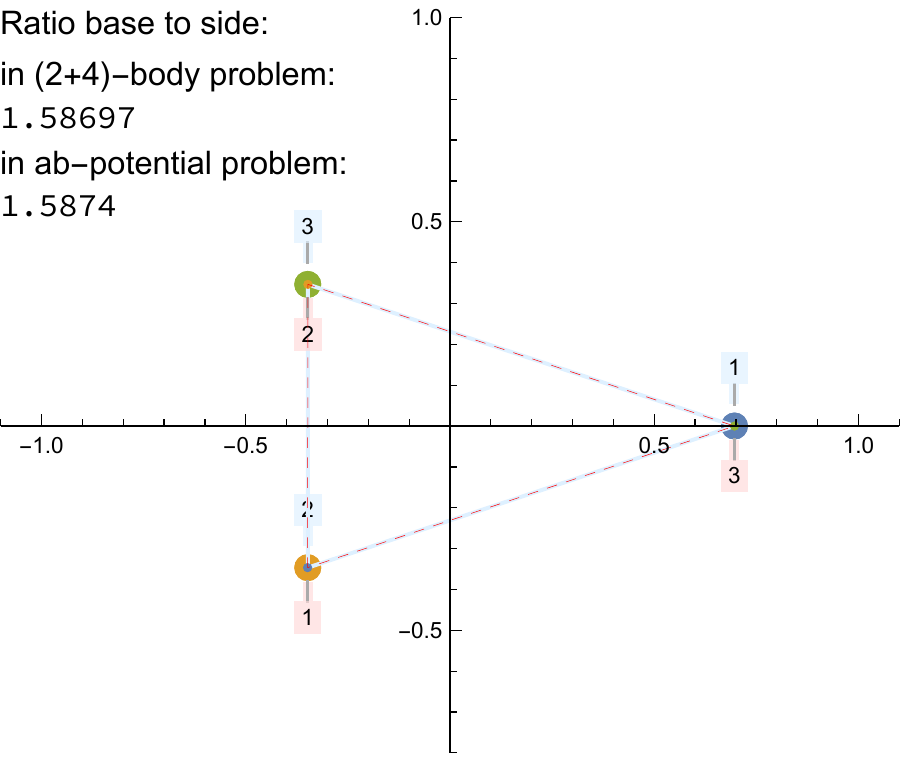}
\end{subfigure}
 \captionof{figure}{Triangle solutions.}\label{fig:triangle3}
\end{figure}

Coordinates of the isosceles triangle in the  (2+3)-body problem, normalized by~(\ref{eqn:normalization}) and coordinates in $ab$-potential are given below:
{
\renewcommand{\arraystretch}{1.5}
$$
\begin{array}{m{1em}|c|rr|rr}
& q_i &\multicolumn{2}{c}{\mbox{PGU normalized coordinates}} & \multicolumn{2}{c}{\mbox{PHU coordinates (in $ab$-potential)}}\\
\hline
\multirow{2}{*}{\rotatebox{90}{triangle~~~}} & q_1 & (-0.696414,& -0.000092914) & (-0.6964118400,& 0.0000000000)\\
 & q_2 & (0.347844,& 0.346726) & (0.3482059200, & 0.3466806372)\\
& q_3 & (0.34857,& -0.346634) & (0.3482059200,& -0.3466806372)
\end{array}
$$
}
\end{itemize}

\subsubsection{$k=4$}
PGU is run for data with two big masses being close to $\frac{1}{2}\left(1-2\cdot 10^{-5} \right)$ and four small masses  $\mu_i=5\cdot 10^{-6}$ for $i=1,\dots,4$; PHU for masses $\tilde{\mu}_i=1/4$ for $i=1,\dots,4$.   The initial boxes for light bodies in PGU are placed at $[-0.2, 0.2]\times [0.8, 0.9]$, first heavy body is placed at $[0.45, 0.55]\times 0$, second is computed by the center of mass.
In both programs PHU and PGU we obtained the same number of solutions and we were able to pair to the solutions from both programs as suggested
by Theorem~\ref{thm:cont}.

We obtained $10 \cdot 4!=240$ solutions,
\begin{itemize}
\item collinear on \OX, $4!$ solutions
\item collinear on \OY, $4!$ solutions
\item isosceles triangle with fourth point close to the base symmetric with respect to \OX\ and its reflection with respect to \OY\
  (see Fig.~\ref{fig:iso-triangles}), $2 \cdot 4!$ solutions
\item  equilateral triangle inside the isosceles one   symmetric with respect to \OX\ and its reflection with respect to \OY\
  (see Fig.~\ref{fig:iso-equi}), $2 \cdot 4!$ solutions
\item rhombus symmetric with respect to the $x$ and $y$ axes, (see Fig.~\ref{fig:rhombus}), $4!$ solutions
\item rectangle symmetric with respect to the $x$ and $y$ axes (see Fig.~\ref{fig:rectangle}), $4!$ solutions
\item `slanted' rhombus and its symmetric image coordinate axes (see Fig.~\ref{fig:slanted-rhombus}), $2\cdot 4!$  solutions
\end{itemize}

\subsubsection*{Comparison of solutions}
\label{subsubsec:full-list}
\begin{itemize}
\item horizontally and vertically collinear (two solutions)
\addtocounter{figure}{-1}  
\begin{figure}[H]
\centering
\begin{subfigure}{0.42\textwidth}
  \centering
\includegraphics[scale=0.7]{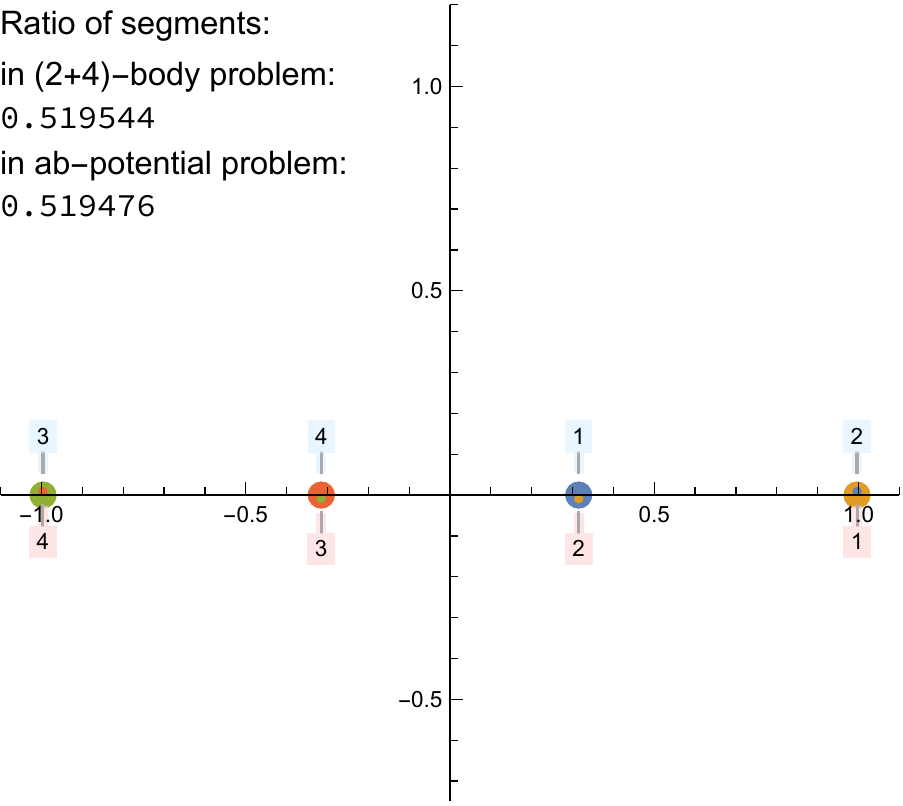}
\end{subfigure}%
\qquad
\begin{subfigure}{0.42\textwidth}
  \centering
\includegraphics[scale=0.7]{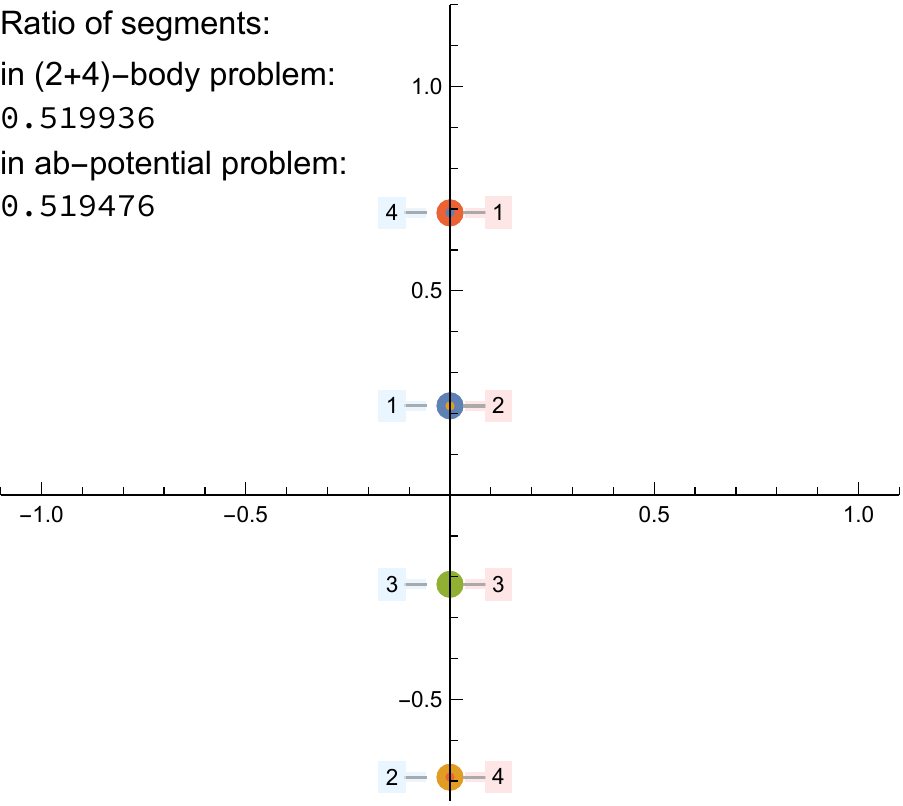}
\end{subfigure}
 \captionof{figure}{Collinear solutions.}\label{fig:coll4}
\end{figure}

\item  isosceles triangle
\addtocounter{figure}{-1}  
\begin{figure}[H]
\centering
\begin{subfigure}{0.42\textwidth}
  \centering
\includegraphics[scale=0.7]{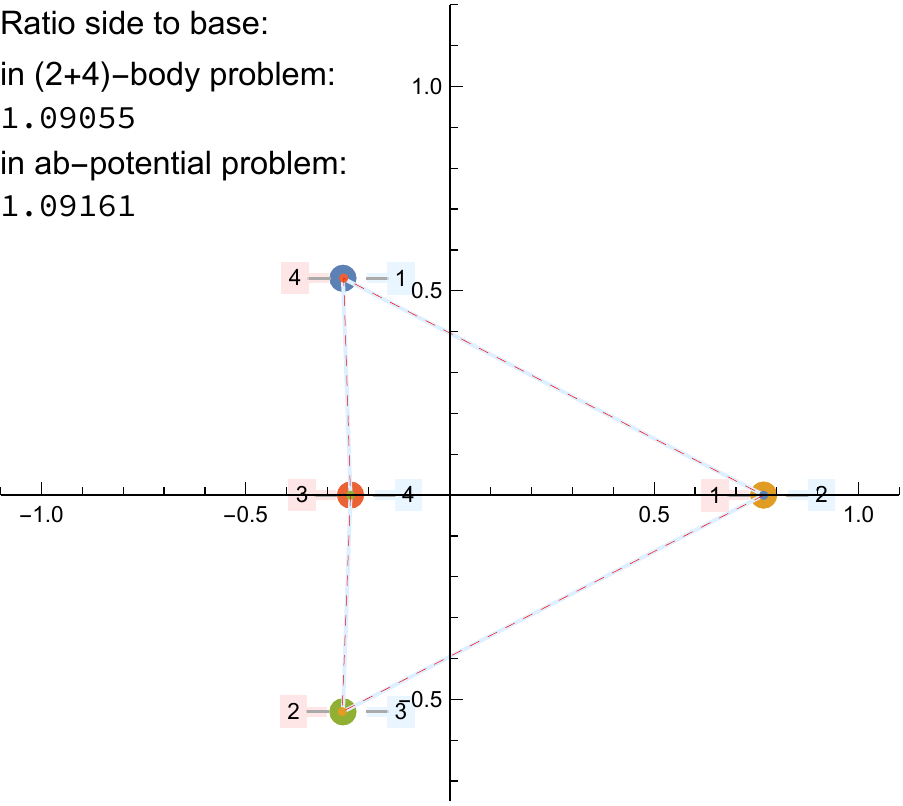}
\end{subfigure}%
\qquad
\begin{subfigure}{0.42\textwidth}
  \centering
\includegraphics[scale=0.7]{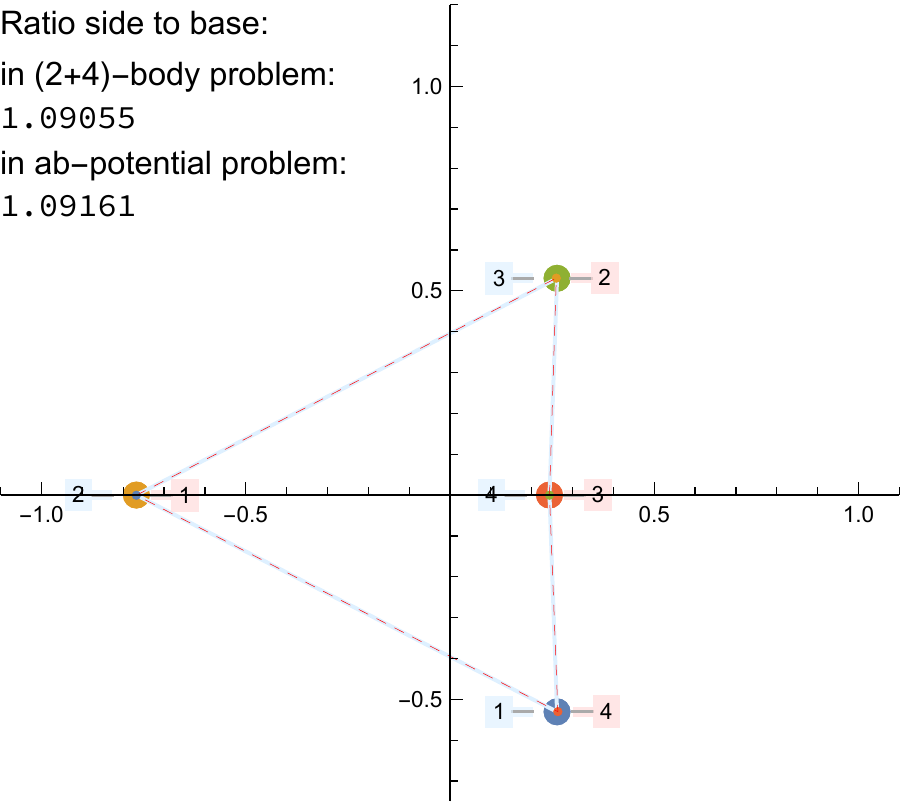}
\end{subfigure}
 \captionof{figure}{Isosceles triangles.}\label{fig:iso-triangles}
\end{figure}

\item equilateral triangle inside the isosceles one
\addtocounter{figure}{-1}  
\begin{figure}[H]
\centering
\begin{subfigure}{0.42\textwidth}
  \centering
\includegraphics[scale=0.7]{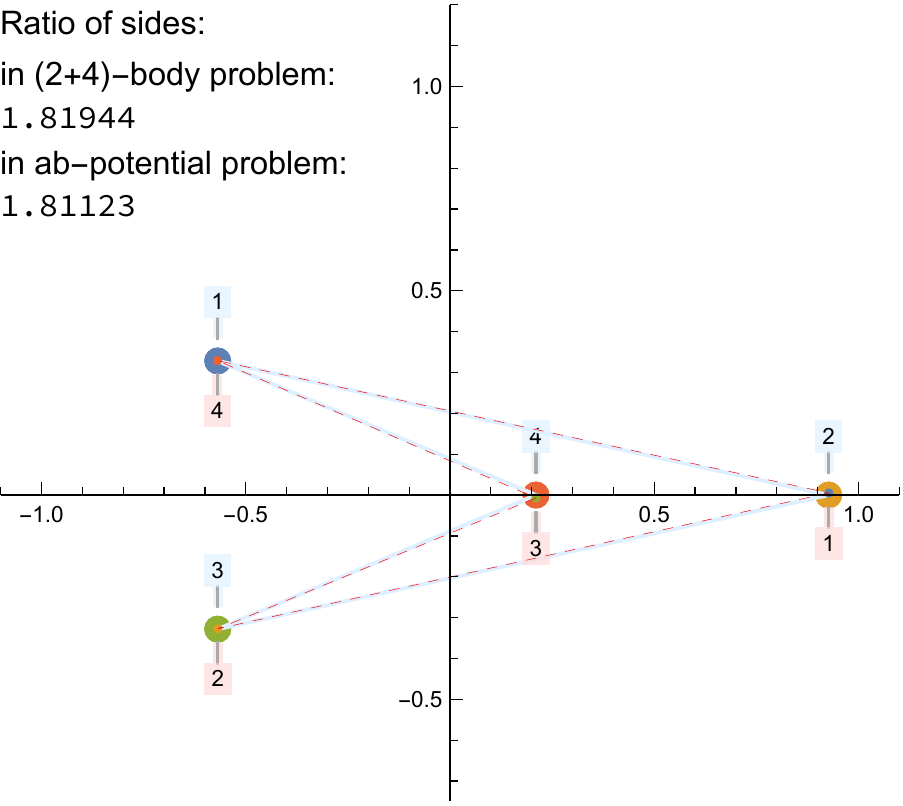}
\end{subfigure}%
\qquad
\begin{subfigure}{0.42\textwidth}
  \centering
\includegraphics[scale=0.7]{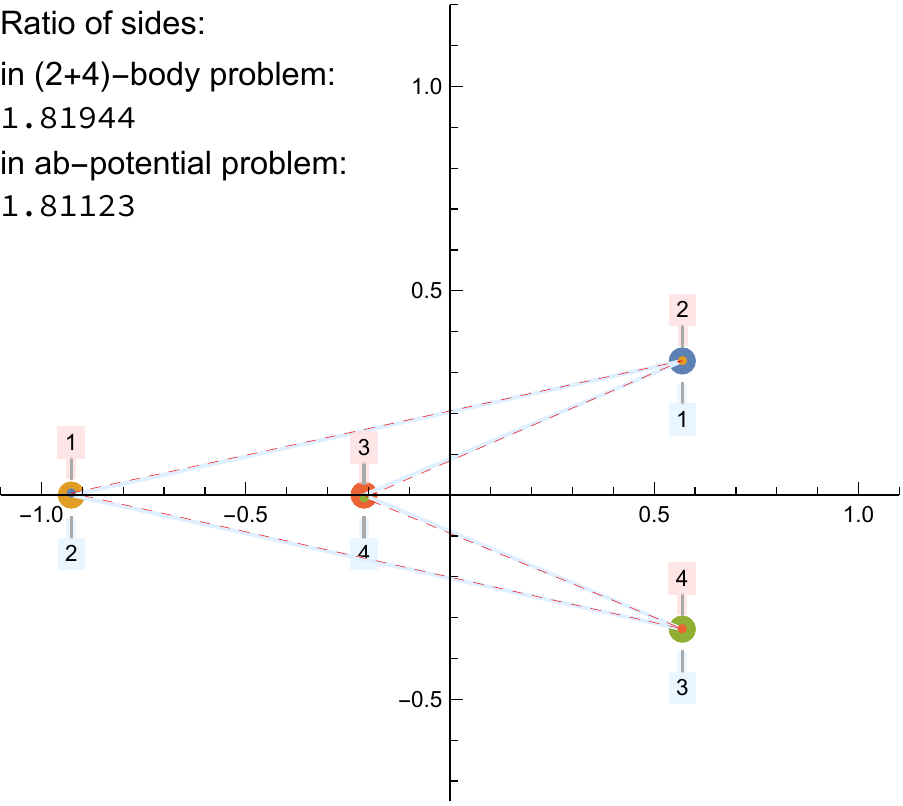}
\end{subfigure}
 \captionof{figure}{Equilateral triangle inside the isosceles one.}\label{fig:iso-equi}
\end{figure}

\item rhombus symmetric with respect to the $x$ and $y$ axes
\begin{figure}[H]
\centering
\includegraphics[scale=0.7]{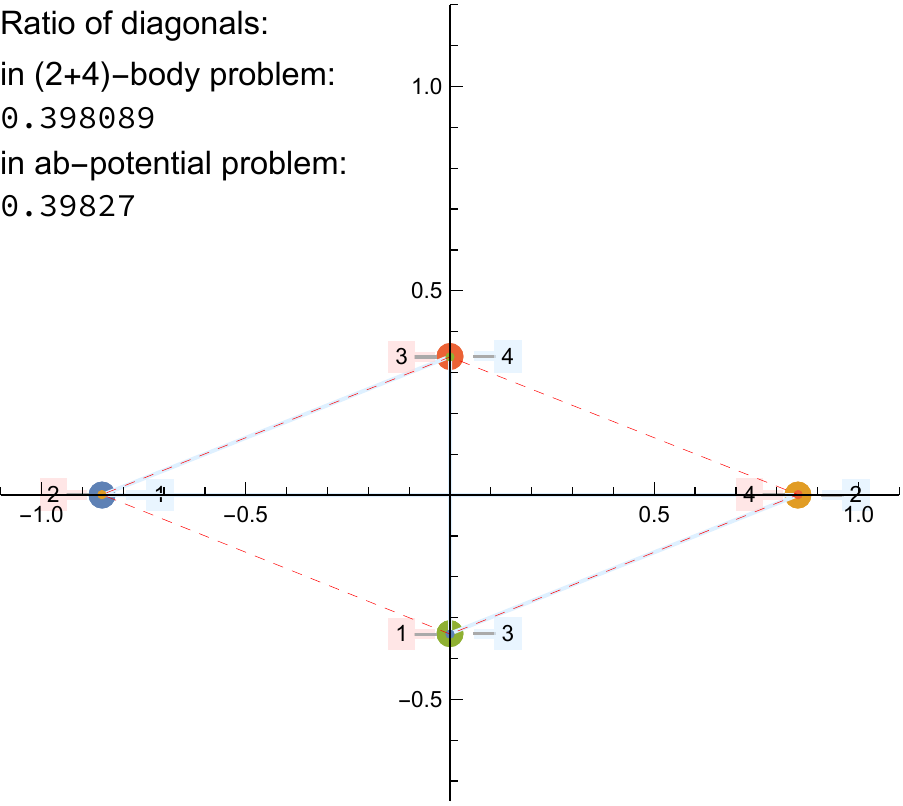}
 \captionof{figure}{Rhombus solution.}\label{fig:rhombus}
\end{figure}

\item rectangle symmetric with respect to the $x$ and $y$ axes
\begin{figure}[H]
\centering
\includegraphics[scale=0.7]{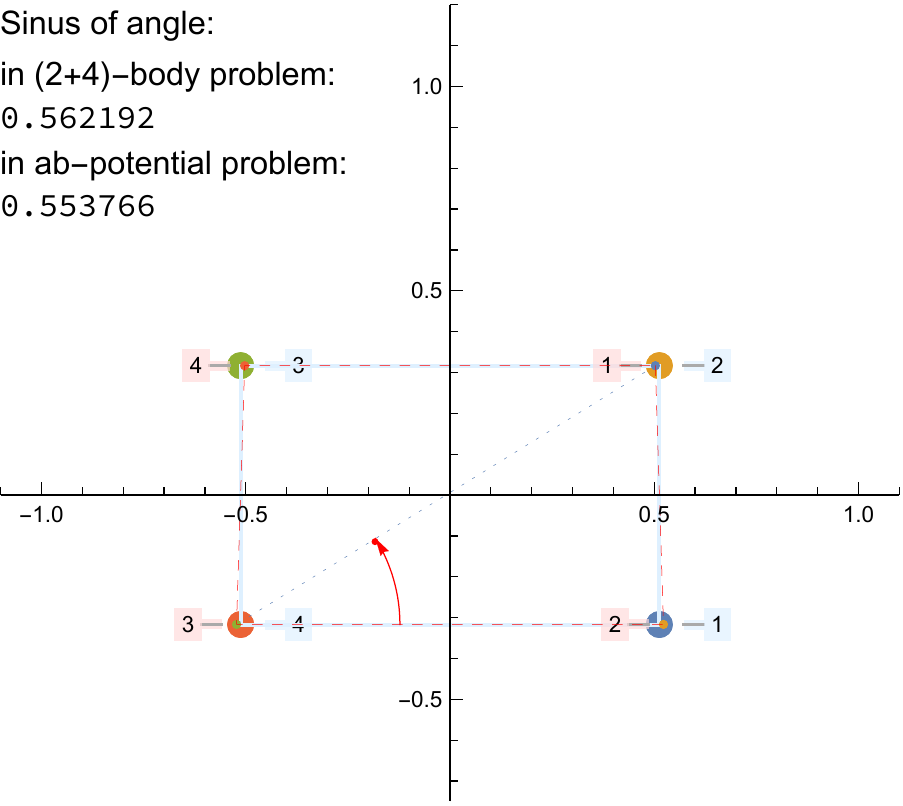}
 \captionof{figure}{Rectangle.}\label{fig:rectangle}
\end{figure}

\item `slanted' rhombus
\addtocounter{figure}{-1}  
\begin{figure}[H]
\centering
\begin{subfigure}{0.42\textwidth}
  \centering
\includegraphics[scale=0.7]{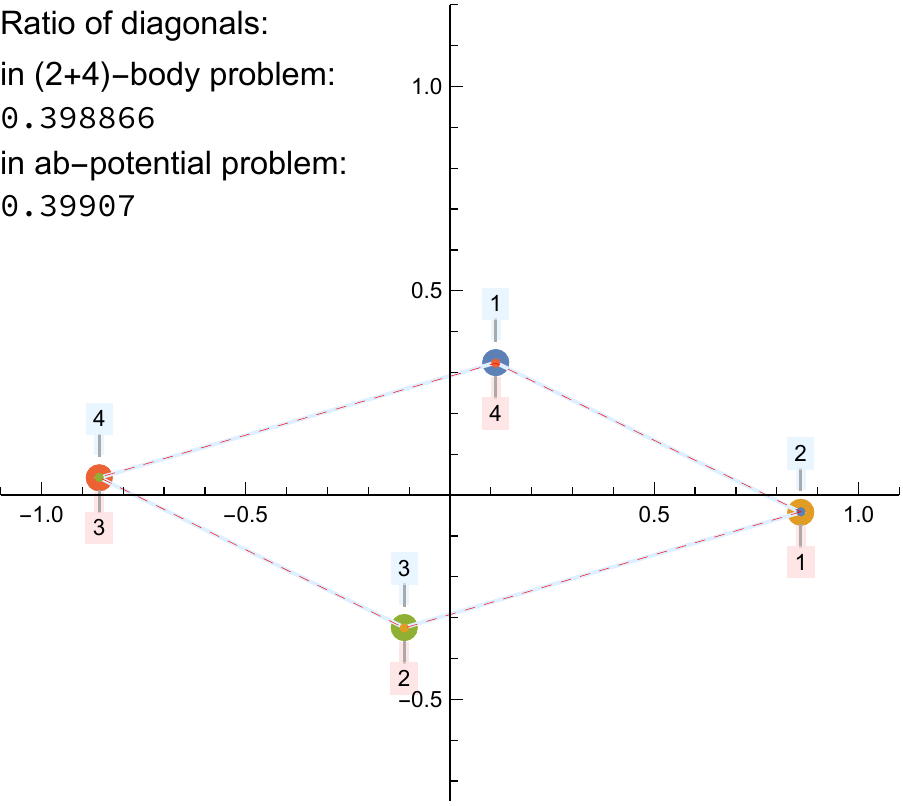}
\end{subfigure}%
\qquad
\begin{subfigure}{0.42\textwidth}
  \centering
\includegraphics[scale=0.7]{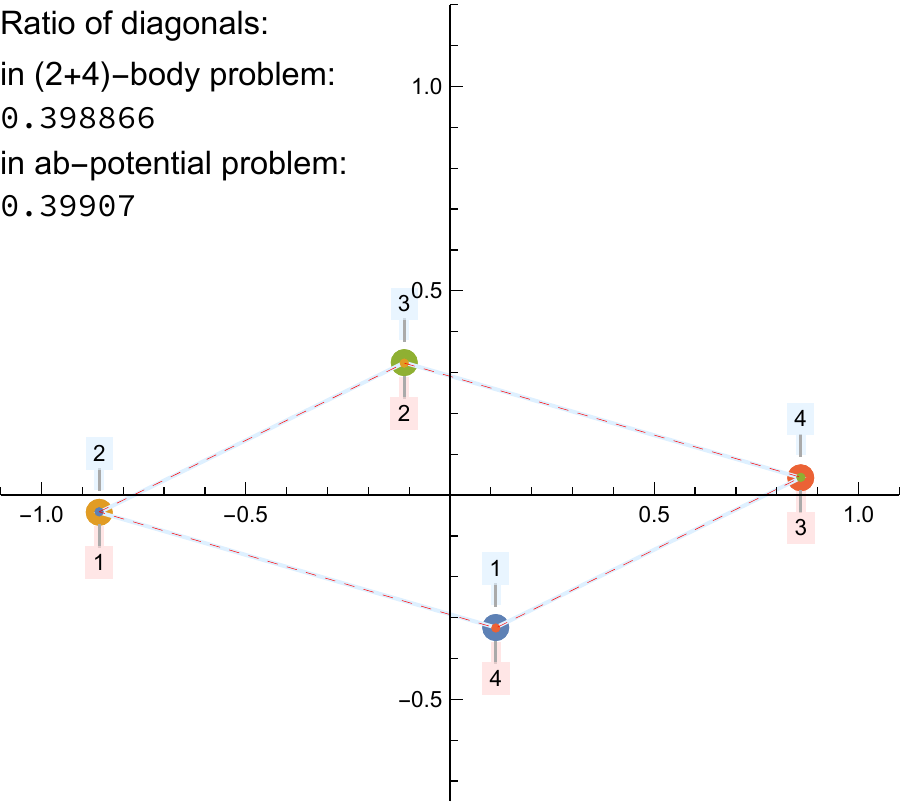}
\end{subfigure}
 \captionof{figure}{`Slanted' rhombuses.}\label{fig:slanted-rhombus}
\end{figure}

\end{itemize}

\subsubsection*{Detailed comparison for symmetrical rhombus}
From Theorem~\ref{thm:rhombus}
we obtain the ratio of diagonals $k = 0.39827$. From PGU
we get $k = 0.398089$ and from PHU
--- $k = 0.39827$.

Coordinates of the rhombus in the (2+4)-body problem, normalized by~(\ref{eqn:normalization}), and coordinates in $ab$-potential are given below:
{$$
\renewcommand{\arraystretch}{1.5}
\begin{array}{m{1em}|c|rr|rr}
& q_i &\multicolumn{2}{c}{\mbox{PGU normalized coordinates}} & \multicolumn{2}{c}{\mbox{PHU coordinates (in $ab$-potential)}}\\
\hline
& q_1 & (0.0000117898,&  -0.340306) & (0.0000000000, & -0.3392209571)\\
\multirow{2}{*}{\rotatebox{90}{rhombus}} & q_2 & (-0.852035,& 0.00109155) & (-0.8517357111,& -0.0000000000)\\
& q_3 & (-0.0000117106,& 0.338064) & (-0.0000000000,& 0.3392209571)\\
& q_4 & (0.852034,& 0.00115058) & (0.8517357111,& 0.0000000000)
\end{array}
$$}

\subsubsection*{Detailed comparison for rectangle}
From Theorem~\ref{thm:rectangle},
the angle between diagonal and \OX\ is $\varphi = 0.553766$. From PGU
we get $\varphi = 0.562192$ and from PHU
--- $\varphi = 0.553766$.

Normalized coordinates of the recatngle in the (2+4)-body problem and coordinates in $ab$-potential are given below:
{$$
\renewcommand{\arraystretch}{1.5}
\begin{array}{m{1em}|c|rr|rr}
& q_i &\multicolumn{2}{c|}{\mbox{PGU normalized coordinates}} &  \multicolumn{2}{c}{\mbox{PHU coordinates (in $a,b$-potential)}}\\
\hline
& q_1 & (0.502676, & 0.316689) &  (0.5124981464, & 0.3168767565)\\
\multirow{2}{*}{\rotatebox{90}{rectangle}} & q_2 & (0.522598, & -0.316654) & (0.5124981464,  &-0.3168767565)\\
& q_3 & (-0.522576, & -0.31669) & (-0.5124981464, & -0.3168767565)\\
& q_4 & (-0.502698, & 0.316654) & (-0.5124981464, & 0.3168767565)
\end{array}
$$}
In fact, in PGU the solution is a trapeziod, very close to being rectanglar.


\appendix
\section{Technical proofs omitted in the main paper}\label{proofs}
\subsection{Proof of Lemma~\ref{lem:ccab-mofin}}\label{inertia-proof}
We show the identity~(\ref{eqn:ax}) --- the proof of (\ref{eqn:by}) is analogous.\\[1ex]

Notice that
$$
\begin{array}{lclp{7cm}}
x_i^2 & = & x_i\left(x_i - \sum_j \mu_jx_j \right) & \quad\quad (since (\ref{eq:cc-cofmass}) implies $\sum_i\mu_ix_i = 0$)\\
  & = & x_i\left(x_i\sum_j \mu_j - \sum_j \mu_jx_j \right) & \quad\quad (use (\ref{eq:sum-mui}))\\
  & = & x_i\left(\sum_j \mu_j(x_i - x_j) \right) &
\end{array}
$$

Now we obtain
\begin{eqnarray*}
\sum_i \mu_ix_i^2
& = & \sum_i\sum_j \mu_i\mu_j(x_i^2 - x_ix_j) \\
  & = & \sum_{i<j} \mu_i\mu_j(x_i^2 - 2x_ix_j + x_j^2) \\
  & = & \sum_{i<j} \mu_i\mu_j(x_i - x_j)^2
  \end{eqnarray*}

On the other hand
\begin{eqnarray*}
\sum_i \mu_iax_i^2 & = &  \sum_{i} \sum_{j, j\neq i}\frac{\mu_i\mu_j(x_i - x_j)x_i}{r_{ij}^3} \quad\quad  (\mbox{by}~(\ref{eqn:a-b}))\\
& = & \sum_{i<j} \frac{\mu_i\mu_j}{r_{ij}^3}(x_i - x_j)^2
\end{eqnarray*}

From~(\ref{eqn:ax}, \ref{eqn:by}) the equation~(\ref{eq:aplusb}) is obvious:
\begin{eqnarray*}
a\left(\sum_i \mu_ix_i^2 \right) + b\left(\sum_i \mu_iy_i^2 \right) & = & \sum_{i<j} \frac{\mu_i\mu_j}{r_{ij}^3}(x_i - x_j)^2  + \sum_{i<j} \frac{\mu_i\mu_j}{r_{ij}^3}(y_i - y_j)^2 \\
& = & \sum_{i<j} \frac{\mu_i\mu_j}{r_{ij}}
\end{eqnarray*}
\qed

\subsection{Proof of Lemma~\ref{lm:lower-bnd}}\label{lower-bound-proof}
From Lemma~\ref{lem:ccab-mofin} we have
\begin{eqnarray*}
\max\{a, b\}\sum_i \mu_i(x_i^2 + y_i^2) & \geq &
a\left(\sum_i \mu_ix_i^2 \right) + b\left(\sum_i \mu_iy_i^2 \right)\\
& = & \sum_{i<j} \frac{\mu_i\mu_j}{r_{ij}} >  \frac{\mu_i\mu_j}{r_{ij}}
\end{eqnarray*}
From the above using (\ref{eq:sum-mui}) we obtain
\begin{eqnarray*}
r_{ij} & > & \frac{\mu_i\mu_j}{\max\{a, b\}\sum_i \mu_i(x_i^2 + y_i^2)}\\
\\
 & > & \frac{\mu_i\mu_j}{\max\{a, b\}R^2\sum_i \mu_i}\\
\\
 & = & \frac{\mu_i\mu_j}{\max\{a, b\}R^2}.
\end{eqnarray*}

\qed

\subsection{Proof of Theorem~\ref{thm:upp-bnd}}\label{upp-thm-proof}
For  the proof of Theorem~\ref{thm:upp-bnd} we will need the following result.

\begin{lemma}
\label{lem:upp-bnd}
Let
\begin{equation}
  M=\sum_i \mu_i.
\end{equation}
Assume that $a>0$ or $b>0$.
Assume $\pkm$ is a solution of~(\ref{eqn:a-b}).
Let  $R=|x_{i_0}|=\max_{i} |x_i|$. If $a>0$, then  for all  $\varepsilon  \in \left(0,R/(n-1)\right)$ holds
\begin{equation}
  R - (n-2) \varepsilon  < \frac{M}{a \varepsilon ^2}.  \label{eq:x-max-dist}
\end{equation}

Analogously, let  $R=|y_{i_0}|=\max_{i} |y_i|$. If $b>0, then$  for all  $\varepsilon  \in \left(0,R/(n-1)\right)$ holds
\begin{equation}
  R - (n-2) \varepsilon  < \frac{M}{b \varepsilon ^2}. \label{eq:y-max-dist}
\end{equation}
\end{lemma}
\begin{proof}
Let us fix any $\varepsilon \in \left(0,R/(n-1)\right)$.
We can assume that
\begin{equation}
p_{i_0}=(R,y_{i_0}).  \label{eq:maxqi-on-OX}
\end{equation}
Let $\mathcal{C}$\index{$\mathcal{C}$} be a minimal subset (cluster)\index{cluster} of indices of bodies satisfying the following conditions
\begin{itemize}
\item $i_0 \in \mathcal{C}$
\item if $j \in \mathcal{C}$ and $|p_k - p_j| \leqslant  \varepsilon $, then $k \in \mathcal{C}$
\end{itemize}
The cluster $\mathcal{C}$ can be constructed as follows: We  start with $i_0 \in \mathcal{C}$. Then we add all bodies which are not farther than $\varepsilon $
from the bodies already in $\mathcal{C}$. We repeat this until the set $\mathcal{C}$ stabilizes, which happens after at most $n-1$ steps.
From  assumption about $\varepsilon$ and $R$ it follows that
\begin{equation}
  R > (n-1)\varepsilon . \label{eq:R>neps}
\end{equation}
Observe that (\ref{eq:R>neps}) implies that
$\mathcal{C} \neq \{1,\dots,n\}$. Indeed (\ref{eq:R>neps}) and (\ref{eq:maxqi-on-OX}) imply that $x_i >0$ for all $i \in \mathcal{C}$. This and the center of mass condition (\ref{eq:cc-cofmass}) implies that $\mathcal{C}$ cannot contain all bodies.
This implies that  the process of building $\mathcal{C}$ must stop after at most $n-2$ steps.
Therefore we obtained a cluster $\mathcal{C}$ with the following properties
\begin{eqnarray}
  p_i &\in& \overline{B(p_{i_0},(n-2)\varepsilon )}, \quad \forall i \in \mathcal{C}, \label{eqn:cluster-ball}\\
  |p_i - p_j| &>& \varepsilon, \quad i \in \mathcal{C}, j \notin \mathcal{C}.
\end{eqnarray}

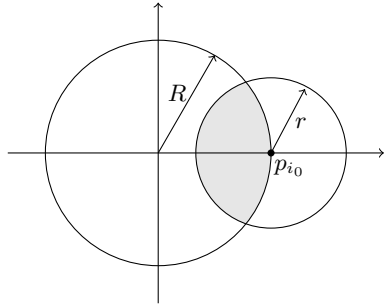
\begin{figure}[htb]
  \centering
\begin{tikzpicture}
\begin{scope}
\clip (3.5,2.0) circle(1.0);           \clip (2.0,2.0) circle(1.5);
\fill[gray!20] (2.5,0.0) rectangle (3.5,4.0);
\end{scope}

\draw[->] (0.0, 2.0) -- (5.0, 2.0);
\draw[->] (2.0, 0.0) -- (2.0, 4.0);

\draw[] (2.0, 2.0) circle (1.5cm);
\draw[->] (2.0, 2.0) -- (2.75, 3.3);

\draw[fill] (3.5, 2.0) circle (0.4mm);
\node[] at (3.75,1.8) {\small $p_{i_0}$};

\draw[] (3.5, 2.0) circle (1.0cm);
\draw[->] (3.5, 2.0) -- (3.95, 2.85);

\node[] at (2.25,2.8) {\small $R$};
\node[] at (3.9,2.4) {\small $r$};

\end{tikzpicture}
\caption{$R > (n-1)\varepsilon$ and $r = (n-2)\varepsilon$; the darkened area is the region where all the bodies from cluster are located.}\label{fig:cluster}
\end{figure}

Note that
\begin{equation*}
\sum_{\substack{i,j \in \mathcal{C}\\  i \neq j}} \frac{\mu_i\mu_j}{r_{ij}^3}(x_i - x_j) = 0
\end{equation*}
thus by adding equations for $x$'s in~(\ref{eqn:a-b})
for $i\in \mathcal{C}$ we obtain
\begin{eqnarray}
  a\sum_{i \in \mathcal{C}}\mu_i x_i  & =  & \sum_{\substack{i,j \in \mathcal{C}\\  i \neq j}}\frac{\mu_i \mu_j (x_i - x_j)}{r_{ij}^3}
  +   \sum_{\substack{i\in \mathcal{C}\\  j\not\in \mathcal{C}}}\frac{\mu_i \mu_j (x_i - x_j)}{r_{ij}^3}. \label{eq:sum-s-eq}
\end{eqnarray}

Let
\begin{eqnarray*}
  \mu_{\mathcal{C}}&=&\sum_{i \in \mathcal{C}} \mu_i, \\
  c_{\mathcal{C}}&=& \frac{1}{\mu_{\mathcal{C}}}\sum_{i \in \mathcal{C}} \mu_i x_i, \\
  F_{\mathcal{C}}&=& \frac{1}{\mu_{\mathcal{C}}}\sum_{\substack{i\in \mathcal{C}\\  j\not\in \mathcal{C}}} \frac{\mu_i \mu_j (x_i - x_j)}{r_{ij}^3}.
\end{eqnarray*}

Observe that (\ref{eq:sum-s-eq}) could be now rewritten as
\begin{equation}
   a\cdot c_{\mathcal{C}}=F_{\mathcal{C}}. \label{eq:cc-cluster-eq}
\end{equation}

It is easy to see (cf.~(\ref{eqn:cluster-ball})) that for $i\in\mathcal{C}$: $x_i \geqslant  R - (n-2) \varepsilon > 0$, hence
\begin{eqnarray*}
  |c_{\mathcal{C}}|  & = & \frac{1}{\mu_{\mathcal{C}}}\sum_{i \in \mathcal{C}} \mu_i x_i, \nonumber\\
  &\geqslant &  R - (n-2) \varepsilon ,   \label{eq:estm-Cs}
\end{eqnarray*}
and
\begin{eqnarray*}
|F_{\mathcal{C}}| &\leqslant &
 \frac{1}{\mu_{\mathcal{C}}}\sum_{\substack{i\in \mathcal{C}\\  j\not\in \mathcal{C}}}  \frac{\mu_i \mu_j }{r_{ij}^2} \leqslant
 \frac{1}{\varepsilon ^2} \frac{1}{\mu_{\mathcal{C}}}\sum_{\substack{i\in \mathcal{C}\\  j\not\in \mathcal{C}}}  \mu_i \mu_j\\
  & = & \frac{1}{\varepsilon ^2} \frac{1}{\mu_{\mathcal{C}}}\left(\sum_{i\in \mathcal{C}} \mu_i \right) \cdot \left(\sum_{j \notin \mathcal{C}}   \mu_j \right)=
  \frac{1}{\varepsilon ^2} \sum_{j \notin \mathcal{C}}   \mu_j <  \frac{M}{\varepsilon ^2}.
\end{eqnarray*}
Hence from the above and (\ref{eq:cc-cluster-eq}) we obtain
\begin{eqnarray*}
  a(R - (n-2) \varepsilon)  \leqslant  |a\cdot c_{\mathcal{C}}| = |F_{\mathcal{C}}| < \frac{M}{\varepsilon ^2}.
\end{eqnarray*}
This completes the proof of~(\ref{eq:x-max-dist}). The proof of (\ref{eq:y-max-dist}) is analogous.

\end{proof}

\textbf{Proof`of Theorem~\ref{thm:upp-bnd}}
We focus on the estimate on $R=\max_i |x_i|$. From Lemma~\ref{lem:upp-bnd} we have the following implication for any $\varepsilon$:
\begin{center}
if $(n-1)\varepsilon < R $, then $R < (n-2)\varepsilon + \frac{M}{a \varepsilon^2}$.
\end{center}
Therefore
\begin{equation}
  R \leq g(\varepsilon):=\max \left((n-1)\varepsilon,  (n-2)\varepsilon + \frac{M}{a \varepsilon^2} \right).
\end{equation}

We look for minimum of $g(\varepsilon)$.
It is easy to see that
\begin{eqnarray*}
  g(\varepsilon)&=&(n-1)\varepsilon, \quad \mbox{for $\varepsilon \geq \left(\frac{M}{a}\right)^{1/3}$}, \\
  g(\varepsilon)&=& g_2(\varepsilon):= (n-2)\varepsilon + \frac{M}{a \varepsilon^2}, \quad \mbox{for $\varepsilon \leq \left(\frac{M}{a}\right)^{1/3}$}
\end{eqnarray*}

The function $g_2(\varepsilon)$ has a global minimum at $\varepsilon_0=\left(\frac{2M}{(n-2)a} \right)^{1/3}$ and
 $$g(\varepsilon_0)=\left( \frac{M}{a}\right)^{1/3}\left(2^{1/3}+2^{-2/3} \right) (n-2)^{2/3}$$

Now $\varepsilon_0 \leq \left(\frac{M}{a}\right)^{1/3}$ iff $n \geq 4$, therefore
\begin{eqnarray*}
  \min g(\varepsilon )=(n-1)\left(\frac{M}{a}\right)^{1/3}, \quad \mbox{for $n \geq 2$}, \\
  \min g(\varepsilon )=\left( \frac{M}{a}\right)^{1/3}\left(2^{1/3}+2^{-2/3} \right) (n-2)^{2/3}, \quad  \mbox{for $n \geq 4$}.
\end{eqnarray*}

\qed

\subsection{Proof of Theorem~\ref{thm:coll-on-coord-axis}}\label{coll-proof}
We divide by $\mu_i$ the equations~(\ref{eqn:a-b})
and  take differences for $x_i$'s and $y_i$'s separately.  Then we obtain system~(\ref{eqn:a-b})
written in matrix form as
\begin{eqnarray}\label{eqn:A-matrix}
A\begin{bmatrix}
x_1 - x_2\\
x_2 - x_3\\
\ldots\\
x_{k-1} - x_k
\end{bmatrix}
 =
 a\begin{bmatrix}
x_1 - x_2\\
x_2 - x_3\\
\ldots\\
x_{k-1} - x_k
\end{bmatrix}
 & \quad\mbox{and}\quad &
A\begin{bmatrix}
y_1 - y_2\\
y_2 - y_3\\
\ldots\\
y_{k-1} - y_k
\end{bmatrix}
 =
 b\begin{bmatrix}
y_1 - y_2\\
y_2 - y_3\\
\ldots\\
y_{k-1} - y_k,
\end{bmatrix}
\end{eqnarray}
where $A\in \reals^{(k-1)\times (k-1)}$ is   matrix with

\begin{eqnarray*}
A_{ij} & = & -\delta_{ij} - \chi_{[1, i-1]}(j)\left(\sum_{t = 1}^j \frac{\mu_t}{r_{it}^3}\right) + \chi_{[i, n-1]}(j)\left(\sum_{t = j+1}^k \frac{\mu_t}{r_{it}^3}\right)\\
  &  & + \chi_{[1, i]}(j)\left(\sum_{t = 1}^j \frac{\mu_t}{r_{(i+1)t}^3}\right) - \chi_{[i+1, n-1]}(j)\left(\sum_{t = j+1}^k \frac{\mu_t}{r_{(i+1)t}^3}\right),
\end{eqnarray*}
where $\delta_{ij}$ is Kronecker delta and $\chi_{[l,p]}$ is the characteristic function of the interval $[l,p]$, i.e.\
$$\chi_{[l,p]}(j) = \left\{
\begin{array}{ll}
1, & l \leq j\leq p\\
0, & \mbox{otherwise}
\end{array}
\right.
$$

For example, matrix $A$ for $k = 3$ is:
$$
 A= \begin{bmatrix}
        \frac{\mu_1+\mu_2}{r_{12}^3} + \frac{\mu_3}{r_{13}^3},  &  \mu_3\left(\frac{1}{r_{13}^3}- \frac{1}{r_{23}^3} \right) \\
        \mu_1\left(\frac{1}{r_{13}^3} -\frac{1}{r_{12}^3}\right), & \frac{\mu_2+\mu_3}{r_{23}^3} + \frac{\mu_1}{r_{13}^3}  \\
\end{bmatrix}.
$$

Observe from~(\ref{eqn:A-matrix}) we see that  for any solution of (\ref{eqn:a-b}) the $x$-differences and $y$-differences are eigenvectors of $A$ of different eigenvalues $a,b$, respectively or are zero. This with collinearity assumption implies that the solution must be on the coordinate axis. The precise argument goes as follows.

From collinearity it follows that there exist $\lambda_i  \neq 0$, $i=2,\ldots, k-1$ such that
\begin{equation*}
  [x_i - x_{i+1},y_i - y_{i+1}]=\lambda_i  [x_1 - x_2,y_1 - y_2].
\end{equation*}
For the proof it is enough to show that $x_1-x_2=0$ or $y_1-y_2=0$.

We have
\begin{eqnarray*}
 {}[x_1-x_2,x_2-x_3,\dots,x_{k-1}-x_k]&=&(x_1-x_2)[1,\lambda_2,\dots,\lambda_{k-1}], \\
 {}[y_1-y_2,y_2-y_3,\dots,y_{k-1}-y_k]&=&(y_1-y_2)[1,\lambda_2,\dots,\lambda_{k-1}].
\end{eqnarray*}

Let
\begin{equation*}
\Lambda=[1,\lambda_2,\dots,\lambda_{k-1}]^t,
\end{equation*}
thus we have
\begin{eqnarray*}
  A \left((x_1 - x_2) \Lambda\right)= a \left((x_1 - x_2) \Lambda\right), \\
 A \left((y_1 - y_2) \Lambda\right)= b \left((y_1 - y_2) \Lambda\right)
  \end{eqnarray*}
Since $a \neq b$, then either $x_1-x_2=0$ or $y_1-y_2$. Hence we obtain our conclusion.

\qed

\subsection{Proof of Theorem~\ref{thm:iso-triangle}}\label{triangle-proof}

Since we are looking for an isosceles triangle symmetrical with respect to \OX\ and since
the center of mass is at the origin, we can assume that
\begin{equation*}
  p_1=(x,0), \quad p_2=\left(-\frac{x}{2},y\right), \quad p_3=\left(-\frac{x}{2},-y\right)
\end{equation*}
where $x\neq 0$ and $y\neq 0$, since the solution cannot be collinear. Without loss of generality, assume $x, y > 0$.

\begin{minipage}[h]{.95\textwidth}
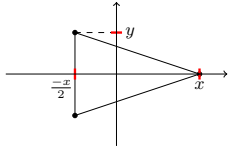
\begin{figure}[H]
  \centering
\resizebox{0.2\textwidth}{!}{%
\begin{tikzpicture}

\draw[fill] (1.5, 0.0) circle (0.4mm);   // p2
\draw[fill] (-0.75, 0.75) circle (0.4mm); // p3
\draw[fill] (-0.75, -0.75) circle (0.4mm);  // p4

\draw[red, very thick] (1.5, -0.1) -- (1.5, 0.1);
\node[] at (1.5, -0.21) {\small $x$};

\draw[] (-0.75, 0.75) -- (1.5, 0.0);  
\draw[] (-0.75, -0.75) -- (1.5, 0.0);  
\draw[] (-0.75, -0.75) -- (-0.75, 0.75);  

\draw[red, very thick] (-0.75, -0.1) -- (-0.75, 0.1);
\node[] at (-1.0, -0.25) {\small $\frac{-x}{2}$};

\draw[dashed] (-0.75, 0.75) -- (0.0, 0.75);
\draw[red, very thick] (-0.1,0.75) -- (0.1, 0.75);
\node[] at (0.25, 0.75) {\small $y$};

\draw[->, thin] (-2.0, 0.0) -- (2.0, 0.0);
\draw[->, thin] (0.0, -1.3) -- (0.0, 1.3);

\end{tikzpicture}
}
\caption{A sample isosceles triangle with $x, y>0$.}\label{fig:triangle}
\end{figure}
\end{minipage}

Then we have
\begin{equation}
r = r_{23}=2y, \quad s = r_{12}=r_{13}=\sqrt{\frac{9 x^2}{4} + y^2}. \label{eq:trowram-rij}
\end{equation}

From the above, it is easy to see that the system~(\ref{eqn:a-b})
can be reduced to
\begin{subequations}
\begin{align}
  a x &= \frac{3\mu x}{r_{12}^3}, \label{eq:3btx} \\
  b y &= \frac{\mu y}{r_{12}^3} + \frac{2\mu y}{r_{23}^3}. \label{eq:3bty}
\end{align}
\end{subequations}

Since $x \neq 0$ and $y \neq 0$, we obtain
\begin{eqnarray*}
a  &=& \frac{3\mu}{r_{12}^3}, \label{eq:3bsolx} \\
b &=& \frac{\mu}{r_{12}^3} + \frac{2\mu}{r_{23}^3}. \label{eq:3bsoly}
\end{eqnarray*}

Hence
\begin{subequations}
\begin{align}
  r_{12}^3 & =\frac{3\mu}{a}, \label{eq:trr-r12}\\
  r_{23}^3 & =\frac{2\mu}{b - \frac{a}{3}}, \label{eq:trr-r23}
  \end{align}
\end{subequations}
For the equation~(\ref{eq:trr-r23}) to make sense, the denominator must be positive ($r = r_{23} > 0$), thus
\begin{equation}
b > \frac{a}{3}.\label{eq:b>a/3}
\end{equation}

Note that~(\ref{eq:b>a/3}) is compatibile with the assumtion~(\ref{eqn:assumption}).
From~(\ref{eq:trowram-rij}) we compute
\begin{subequations}
\begin{align}
  y^3 & =\frac{\mu}{4(b - \frac{a}{3})},  \label{eq:trr-y3}\\
x^2  & = \left(\frac{2}{3}\right)^2(r_{12}^2 - y^2).
\end{align}
\end{subequations}
To have a real solution, we demand $r_{12}^2 - y^2 >0$, which implies~(\ref{eqn:assumption}).
\qed

\subsection{Proof of Theorem~\ref{thm:rhombus}}\label{rhombus-proof}
Assume that the rhombus is as on the figure below:
\begin{equation*}
p_1 = (x, 0)\quad p_2 = (0, y)\quad p_3 = (-x, 0)\quad p_4 = (0, -y),
\end{equation*}
where $x >0$ and $y > 0$.
Ratio of diagonals is $k = y/x > 0$.

\begin{minipage}[h]{.95\textwidth}
\begin{figure}[H]
  \centering
\resizebox{0.2\textwidth}{!}{%
\begin{tikzpicture}
\draw[fill] (0.0, 1.0) circle (0.4mm);  // p1
\draw[fill] (1.5, 0.0) circle (0.4mm);   // p2
\draw[fill] (-1.5, 0.0) circle (0.4mm); // p3
\draw[fill] (0.0, -1.0) circle (0.4mm);  // p4

\node[] at (1.5, 0.2) {\small $p_1$};
\node[] at (0.3, 1.05) {\small $p_2$};
\node[] at (-1.5, 0.2) {\small $p_3$};
\node[] at (0.3, -1.05) {\small $p_4$};

\draw[thick] (0.0, 1.0) -- (1.5, 0.0);  
\draw[thick] (0.0, 0.0) -- (0.0, 1.0);  
\draw[thick] (0.0, 0.0) -- (1.5, 0.0);  
\draw[dashed] (1.5, 0.0) -- (0.0, -1.0);
\draw[dashed] (0.0, -1.0) -- (-1.5, 0.0);
\draw[dashed] (-1.5, 0.0) -- (0.0, 1.0);

\node[] at (0.7, 0.7) {\small $r$};
\node[] at (-0.15, 0.4) {$y$};
\node[] at (0.3, -0.2) {$x$};

\draw[->, thin] (-2.0, 0.0) -- (2.0, 0.0);
\draw[->, thin] (0.0, -1.5) -- (0.0, 1.5);

\end{tikzpicture}
}
\end{figure}
\end{minipage}

It is easy to see that the system~(\ref{eqn:a-b})
can be reduced to the following two equations
\begin{subequations}
\begin{align}
a & =  \mu\frac{2}{r^3} + \mu\frac{1}{4x^3} \label{eqn:qa}\\
b & =  \mu\frac{2}{r^3} + \mu\frac{1}{4y^3}\label{eqn:qb}.
\end{align}\label{eqn:a-b-sol}
\end{subequations}
\quad\hfill


Since $y = kx$, it holds that  $r = x\sqrt{1 + k^2}$ and we can wrtite
\begin{subequations}
\begin{align}
a & =  \frac{\mu}{x^3}\left(\frac{2}{(\sqrt{ (1+ k^2)})^3} + \frac{1}{4}\right) \label{eqn:a-x3}\\
b & =  \frac{\mu}{x^3}\left(\frac{2}{(\sqrt{ (1+ k^2)})^3} + \frac{1}{4k^3}\right).\label{eqn:b-x3}
\end{align}\
\end{subequations}
Dividing the corresponding sides of~(\ref{eqn:a-x3}) and (\ref{eqn:b-x3}) we obtain~(\ref{eqn:k}). The  lhs of (\ref{eqn:k}),
 $$f(k) = \frac{8k^3 + k^3(\sqrt{ 1+ k^2})^3}{(\sqrt{ 1+ k^2})^3 + 8k^3},$$
is strictly increasing, $\lim_{k \to \infty} f(k)=\infty$ and $f(0) = 0$, so equation~(\ref{eqn:k}) has exactly one solution (see Fig.~\ref{fig:f-function}).

By~(\ref{eqn:qa}) we have
\begin{eqnarray}
a & = & \mu\frac{2}{r^3} + \mu\frac{1}{4x^3} >   \mu\frac{2}{r^3} + \mu\frac{1}{4r^3} =  \frac{9\mu}{4r^3},\label{eqn:r-a-bound}
\end{eqnarray}
and analogously from~(\ref{eqn:qb}) we obtain
\begin{equation}
b  > \frac{9\mu}{4r^3}.\label{eqn:r-b-bound}
\end{equation}
Combining~(\ref{eqn:r-a-bound}) and (\ref{eqn:r-b-bound}) we get the lower bound for $r$, namely
\begin{equation}
  r^3 > \max\left(\frac{9\mu}{4a},\frac{9\mu}{4b} \right)=\frac{9 \mu}{4 \min(a,b)}. \label{eq:rlbapp6}
\end{equation}
And again using~(\ref{eqn:a-b-sol}),  we compute $x^2$ and $y^2$:
\begin{eqnarray*}
x^2 & = & r^2\left(\frac{\mu}{4(ar^3 - 2\mu)} \right)^{\frac{2}{3}}\\
y^2 & = & r^2\left(\frac{\mu}{4(br^3 - 2\mu)} \right)^{\frac{2}{3}}.
\end{eqnarray*}
Observe that from (\ref{eq:rlbapp6}) it follows that denominators in the above equations are positive.

Hence
$$
r^2 =  x^2 + y^2 =  r^2\left(\frac{\mu}{4(ar^3 - 2\mu)} \right)^{\frac{2}{3}} +  r^2\left(\frac{\mu}{4(br^3 - 2\mu)} \right)^{\frac{2}{3}}
$$
and we obtain (\ref{eqn:m}).  Using lhs of (\ref{eqn:m}) we define a function
$$
g_1(r) = \left(\frac{2\mu}{ar^3 - 2\mu} \right)^{\frac{2}{3}} +  \left(\frac{2\mu}{br^3 - 2\mu} \right)^{\frac{2}{3}}.
$$
The equation~(\ref{eqn:m}) has exactly one solution, with $r  > \displaystyle{\left(\frac{2\mu}{\min\{a, b\}}\right)^{1/3}}$ (see Fig.~\ref{fig:function-sides})

\begin{figure}[H]
\centering
\begin{subfigure}{0.35\textwidth}
  \centering
\includegraphics[scale=0.45]{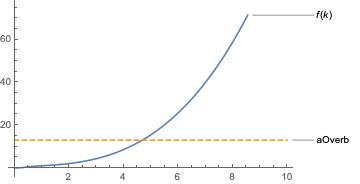}
\captionof{figure}{Function $f$ of a ratio of diagonals $k$. Yellow dashed line is at $a/b$. }\label{fig:f-function}
\end{subfigure}
\qquad
\begin{subfigure}{0.5\textwidth}
  \centering
\includegraphics[scale=0.5]{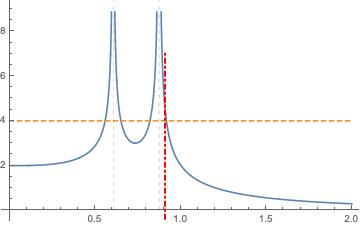}
 \captionof{figure}{Function $g_1$ of a size of sides $r$. Asymptotes (dashed vertical lines) are at $r = (2\mu/b)^{1/3}$ and $r = (2\mu/a)^{1/3}$; the lower bound for $r$ is marked as red line. }\label{fig:function-sides}
\end{subfigure}
\caption{Functions describing geometry of the rhombus.}
\end{figure}
\qed

\subsection{Proof of Theorem~\ref{thm:rectangle}}\label{rectangle-proof}
Assume that the rectangle is as on the figure below:
\begin{equation*}
p_1 = (x, y)\quad p_2 = (x, -y)\quad p_3 = (-x, -y)\quad p_4 = (-x, y),
\end{equation*}
where $x, y >0,$ $x = r\cos(\varphi)$ and $y = r\sin(\varphi)$.

\begin{minipage}[h]{.95\textwidth}
\begin{figure}[H]
  \centering
\resizebox{0.3\textwidth}{!}{%
\begin{tikzpicture}
\draw[fill] (1.6, 1.0) circle (0.4mm);   // p1
\draw[fill] (-1.6, 1.0) circle (0.4mm);  // p2
\draw[fill] (1.6, -1.0) circle (0.4mm); // p3
\draw[fill] (-1.6, -1.0) circle (0.4mm);  // p4

\draw[dashed] (1.6, 1.0) -- (1.6, -1.0);
\draw[dashed] (1.6, -1.0) -- (-1.6, -1.0);
\draw[dashed] (-1.6, -1.0) -- (-1.6, 1.0);
\draw[dashed] (-1.6, 1.0) -- (1.6, 1.0);

\node[] at (1.8, 1.2) {\small $p_1 = (x,y)$};
\node[] at (1.6, -1.2) {\small $p_2$};
\node[] at (-1.6, -1.2) {\small $p_3 = (-x, -y)$};
\node[] at (-1.6, 1.2) {\small $p_4$};

\draw[thick] (1.6, 1.0) -- (-1.6, -1.0);  

\draw[very thin, ->] (1.0,0) arc (0:40:0.75);
\node[] at (0.65, 0.15) {\small $\varphi$};

\node[] at (0.65, 0.6) {\small $r$};
\node[] at (-0.6, -0.6) {\small $r$};

\draw[->, thin] (-2.5, 0.0) -- (2.5, 0.0);
\draw[->, thin] (0.0, -1.5) -- (0.0, 1.5);

\end{tikzpicture}
}
\end{figure}
\end{minipage}

It is easy to see that the system~(\ref{eqn:a-b})
can be reduced to
\begin{subequations}
\begin{align}
  a  &=  \frac{\mu }{4}\left(\frac{1}{r^3} + \frac{1}{x^3}\right),\label{eqn:rec-qa}\\
  b  &=  \frac{\mu }{4}\left(\frac{1}{r^3}  + \frac{1}{y^3}\right).\label{eqn:rec-qb}
\end{align}\label{eqn:rec-a-b-sol}
\end{subequations}
and afterwards to
\begin{eqnarray*}
  a  &=&  \frac{\mu }{4r^3}\left(1 + \frac{1}{\cos^3(\varphi)}\right),\\
  b  &=&  \frac{\mu }{4r^3}\left(1 + \frac{1}{\sin^3(\varphi)}\right).
\end{eqnarray*}
Diving the corresponding sides of above equations we obtain~(\ref{eqn:phi}). The lhs of~(\ref{eqn:phi}),
$$f(\varphi) = \frac{\sin^3(\varphi)(\cos^3(\varphi) + 1)}{\cos^3(\varphi)(\sin^3(\varphi) + 1)},$$
is strictly increasing for $\varphi\in [0, \pi/2)$, $\lim_{\varphi \to \pi/2} f(\varphi)=\infty$ and $f(0) = 0$, so equation~(\ref{eqn:phi}) has unique solution in $[0, \pi/2)$ (see Fig.~\ref{fig:recatngle-function}).

By~(\ref{eqn:rec-a-b-sol}) we have
\begin{eqnarray*}
a  & =  & \frac{\mu}{4}\left(\frac{1}{r^3} + \frac{1}{x^3}\right)
    >  \frac{9\mu}{4r^3}\\
b  & =  & \frac{\mu}{4}\left(\frac{1}{r^3} + \frac{1}{y^3}\right)
    >  \frac{9\mu}{4r^3}
\end{eqnarray*}
Combining above inequalities we get the lower bound for $r$.
Again using~(\ref{eqn:rec-a-b-sol}),  we compute $x^2$ and $y^2$:
\begin{eqnarray*}
x^2 & = & r^2\left(\frac{\mu}{4ar^3 - \mu} \right)^{\frac{2}{3}}\\
y^2 & = & r^2\left(\frac{\mu}{4br^3 - \mu} \right)^{\frac{2}{3}}
\end{eqnarray*}
hence
$$
r^2 =  x^2 + y^2 =  r^2\left(\left(\frac{\mu}{4ar^3 - \mu} \right)^{\frac{2}{3}} + \left(\frac{\mu}{4br^3 - \mu} \right)^{\frac{2}{3}}\right)
$$
and we obtain (\ref{eqn:mu}). Using lhs of (\ref{eqn:m}) define a function
$$
g_2(r) = \left(\frac{\mu}{4ar^3 - \mu} \right)^{\frac{2}{3}} + \left(\frac{\mu}{4br^3 - \mu} \right)^{\frac{2}{3}}.
$$
The equation~(\ref{eqn:mu}) has exactly one solution, with $r  > (\mu/(4a))^{1/3}$ (see Fig.~\ref{fig:function-diagonals})

\begin{figure}[H]
\centering
\begin{subfigure}{0.35\textwidth}
  \centering
\includegraphics[scale=0.3]{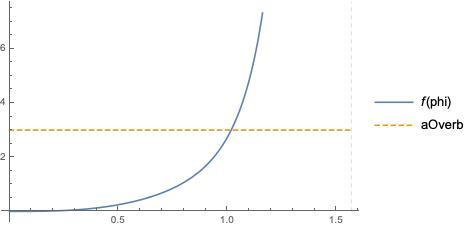}
\captionof{figure}{Function $f(\varphi)$; horizontal dashed line is at $a/b$. Note that, for any $a$, $b$ and $\varphi\in (0, \pi/2)$, the equation $f(\varphi) = a/b$ has exactly one solution. }\label{fig:recatngle-function}
\end{subfigure}
\qquad
\begin{subfigure}{0.5\textwidth}
  \centering
\includegraphics[scale=0.5]{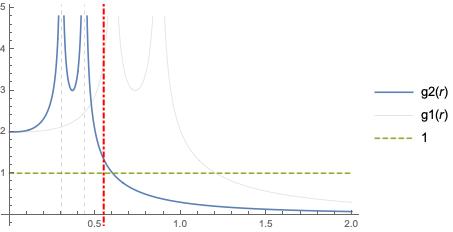}
 \captionof{figure}{Function $g_2(r)$. Asymptotes (dashed vertical lines) are at $r = (\mu/(4a))^{1/3}$ and $r = (\mu/(4b))^{1/3}$; red dotted-dashed line illustrates the lower bound for $r$. Thin gray lines are the graph of function $g_1$ for rhombus.}\label{fig:function-diagonals}
\end{subfigure}
\caption{Functions describing geometry of rectangle.}
\end{figure}
\qed

\subsection{Computation of $a,b$ at $L_4$ in the restricted three body problem}
\label{sec:L4ab}
We assume that point of mass $m_i$ is located at $p_i$, $i=1,2$ where

\begin{eqnarray*}
  p_1=\left(-1/2,0\right), \quad p_2=\left(1/2,0\right),   \label{eq:p1p2p3}
\end{eqnarray*}

The masses are normalized to $m_1 + m_2 = 1$.  The zero-mass (third) body has an
unknown location $z=(x,y)$. The equation for relative equilibria is $\nabla
V(z;m) = 0$, where

\begin{equation*}
V(z;m) = \frac12\|z - c\|^2 + \sum_{i=1}^3 m_i \| z - p_i\|^{-1}
\end{equation*}
where $c=\sum_{i=1}^2 m_i p_i$ is the center of mass.

It is well known  (see for example \cite{MD}) that  we have five relative equilibria: three collinear with the primaries (called $L_1,L_2,L_3$) and two triangular $L_4$ and $L_5$, which form an equilateral triangle with the primaries.  $L_4$ and $L_5$ are mutually symmetric with respect to the line passing through the primaries.
For any choice of positive masses $m_1$, $m_2$, such that  $m_1 + m_2 = 1$, we have $L_4=(0,\sqrt{3}/2)$ and $L_5=(0,-\sqrt{3}/2)$.

Expressions in the lemma below are computed using \emph{Mathematica} program~\cite{Mth}.
\begin{lemma}
\label{lem:sep-eigenval}

\begin{equation*}
D^2 V(L_4)= \begin{bmatrix}
 \frac{3}{4} & \frac{3}{4} \sqrt{3} (2 m_1-1) \\
 \frac{3}{4} \sqrt{3} (2 m_1-1) & \frac{9}{4} \\
\end{bmatrix}
=
\begin{bmatrix}
 \frac{3}{4} & \frac{3}{4} \sqrt{3} ( m_1-m_2) \\
 \frac{3}{4} \sqrt{3} ( m_1-m_2) & \frac{9}{4} \\
\end{bmatrix}
\end{equation*}

The matrix $D^2V(L_4)$ is positive definite.  Its eigenvalues $\lambda_1 > \lambda_2$ are given by
\begin{eqnarray*}
 \lambda_2 & = & \frac{3}{2} \left(1-\sqrt{3 m_1^2-3m_1+1}\right)= \frac{3}{2} \left(1-\sqrt{1-3 m_1 m_2}\right) , \\ 
\lambda_1 & = & \frac{3}{2} \left(\sqrt{3 m_1^2-3
   m_1+1}+1\right)=\frac{3}{2} \left(1+\sqrt{1-3 m_1 m_2}\right)
\end{eqnarray*}
and satisfy the following conditions
\begin{eqnarray*}
   \sqrt{3} &>& \lambda_1 - \lambda_2  \geq \frac{3}{2} \\
   \lambda_1 &<& \frac{3}{2}\left(1+ \frac{1}{\sqrt{3}} \right), \\
   \lambda_2&=&   \frac{9 m_1m_2}{4} + O((m_1m_2)^2).
\end{eqnarray*}
\end{lemma}
\proof
Expressions for $\lambda_{1,2}$  can be checked by direct computation. It remains to establish the remaining inequalities.
Since
\begin{equation*}
  3m_1^2 - 3m_1 + 1= 3\left(\left(m_1-\frac{1}{2}\right)^2 + \frac{1}{12}\right)=1-3m_1 m_2,
\end{equation*}
hence for $m_1 \in (0,1)$ we obtain
\begin{eqnarray*}
  \sqrt{3} > \lambda_1 - \lambda_2 & = & 3 \sqrt{3m_1^2 - 3m_1+1}  \geq \frac{3}{2} \\
  \lambda_1 & < & \frac{3}{2}\left(1+ \frac{1}{\sqrt{3}} \right)  
\end{eqnarray*}

\begin{eqnarray*}
  \lambda_2 & = & \frac{3}{2}\left( 1 - \sqrt{1-3m_1m_2}\right)\\
    & = & \frac{3}{2}\left(1-\left(1-\frac{3 m_1m_2}{2} + O((m_1m_2)^2)\right) \right)\\
    & = &     \frac{9 m_1m_2}{4} + O((m_1m_2)^2).
\end{eqnarray*}
\qed

In particular for the case of equal masses we have
\begin{lemma}
\label{lem:L4eigenVal-eq}
  Assume that $m_1=m_2=1/2$. Then the eigenvalues of $D^2V(L_4)$ are $\frac{3}{4}$ and $\frac{9}{4}$ with the eigenvectors $(1,0)$ and $(0,1)$,
  respectively. 
\end{lemma}

\end{document}